\documentclass[hidelinks,12pt,a4paper]{article}
\usepackage[utf8x]{inputenc}
\usepackage[top=30mm, bottom=30mm, inner=20mm, outer=20mm, headsep=10mm, footskip=12mm]{geometry}
\usepackage{amssymb}
\usepackage{amsmath,bm}
\usepackage{mathtools}
\usepackage{amsthm}
\usepackage[english]{babel}
\usepackage{lmodern}
\usepackage{microtype}
\usepackage[mathscr]{euscript}
\usepackage{zref-clever}
\usepackage[colorlinks=true]{hyperref}
\usepackage[dvipsnames]{xcolor}
\usepackage{todonotes}
\usepackage{appendix}
\usepackage{bm}\usepackage{cancel}
\usepackage{upgreek}
\usepackage[normalem]{ulem}

\usepackage{pgf, tikz, pgfplots}
\pgfplotsset{compat=1.17}
\usetikzlibrary{calc, arrows, shapes, positioning, patterns, angles, quotes, intersections, through, backgrounds}
\usetikzlibrary{decorations.markings}

\numberwithin{equation}{section}

\theoremstyle{plain}
\newtheorem{definition}{Definition}[section]

\newtheorem{theorem}[definition]{Theorem}
\newtheorem{proposition}[definition]{Proposition}
\newtheorem{lemma}[definition]{Lemma}
\newtheorem{corollary}[definition]{Corollary}

\theoremstyle{definition}
\newtheorem{remark}[definition]{Remark}

\let\oldnewtheorem\newtheorem

\RenewDocumentCommand{\newtheorem}{momo}{
\IfValueTF{#2}{
\AddToHook{env/#1/begin}{
\zcsetup{countertype={#2=#1}}}
\zcRefTypeSetup{#1}{
Name-sg = #3 ,
}
\oldnewtheorem{#1}[#2]{#3}
}{
\AddToHook{env/#1/begin}{
\zcsetup{countertype={#1=#1}}}
\zcRefTypeSetup{#1}{
Name-sg = #3 ,
}
\IfValueTF{#4}{
\oldnewtheorem{#1}{#3}[#4]
}{
\oldnewtheorem{#1}{#3}
}
}
}

\newcommand{\Cref}[1]{\zcref[S]{#1}}

\makeatletter
\let\save@mathaccent\mathaccent
\newcommand*\if@single[3]{%
\setbox0\hbox{${\mathaccent"0362{#1}}^H$}%
\setbox2\hbox{${\mathaccent"0362{\kern0pt#1}}^H$}%
\ifdim\ht0=\ht2 #3\else #2\fi
}
\newcommand*\rel@kern[1]{\kern#1\dimexpr\macc@kerna}
\newcommand*\widebar[1]{\@ifnextchar^{{\wide@bar{#1}{0}}}{\wide@bar{#1}{1}}}
\newcommand*\wide@bar[2]{\if@single{#1}{\wide@bar@{#1}{#2}{1}}{\wide@bar@{#1}{#2}{2}}}
\newcommand*\wide@bar@[3]{%
\begingroup
\def\mathaccent##1##2{%
\let\mathaccent\save@mathaccent
\if#32 \let\macc@nucleus\first@char \fi
\setbox\z@\hbox{$\macc@style{\macc@nucleus}_{}$}%
\setbox\tw@\hbox{$\macc@style{\macc@nucleus}{}_{}$}%
\dimen@\wd\tw@
\advance\dimen@-\wd\z@
\divide\dimen@ 3
\@tempdima\wd\tw@
\advance\@tempdima-\scriptspace
\divide\@tempdima 10
\advance\dimen@-\@tempdima
\ifdim\dimen@>\z@ \dimen@0pt\fi
\rel@kern{0.6}\kern-\dimen@
\if#31
\overline{\rel@kern{-0.6}\kern\dimen@\macc@nucleus\rel@kern{0.4}\kern\dimen@}%
\advance\dimen@0.4\dimexpr\macc@kerna
\let\final@kern#2%
\ifdim\dimen@<\z@ \let\final@kern1\fi
\if\final@kern1 \kern-\dimen@\fi
\else
\overline{\rel@kern{-0.6}\kern\dimen@#1}%
\fi
}%
\macc@depth\@ne
\let\math@bgroup\@empty \let\math@egroup\macc@set@skewchar
\mathsurround\z@ \frozen@everymath{\mathgroup\macc@group\relax}%
\macc@set@skewchar\relax
\let\mathaccentV\macc@nested@a
\if#31
\macc@nested@a\relax111{#1}%
\else
\def\gobble@till@marker##1\endmarker{}%
\futurelet\first@char\gobble@till@marker#1\endmarker
\ifcat\noexpand\first@char A\else
\def\first@char{}%
\fi
\macc@nested@a\relax111{\first@char}%
\fi
\endgroup
}
\makeatother

\newcommand{\R}{\mathbb R}

\newcommand{\eps}{\varepsilon}

\usepackage[xcolor]{changebar}
\cbcolor{blue}

\usepackage[inline]{enumitem}
\newcommand{\enumlabelformat}{\roman}

\newlength{\thelabelsep}
\newcounter{inlineenum}
\renewcommand{\theinlineenum}{\enumlabelformat{inlineenum}}

\let\phi\varphi

\DeclareMathOperator{\TC}{TC}
\DeclareMathOperator{\vol}{vol}

\DeclareMathOperator{\arcosh}{arcosh}

\newcommand{\nchi}{{\raise.3ex\hbox{$\chi$}}}

\usepackage[runin]{abstract}

\newcommand{\lm}[1]{\mathbb{L}^2(#1)}

\newcommand{\LpLS}{Lo\-rentz\-ian pre-length space }
\newcommand{\LpLSn}{Lo\-rentz\-ian pre-length space}

\newcommand{\LpLSsn}{Lo\-rentz\-ian pre-length spaces}

\newcommand*{\bp}{\bar{p}}
\newcommand*{\bq}{\bar{q}}
\newcommand*{\br}{\bar{r}}

\newcommand{\tp}{\tilde{p}}
\newcommand{\tq}{\tilde{q}}
\newcommand{\tx}{\tilde{x}}
\newcommand{\tr}{\tilde{r}}

\newcommand{\ty}{\tilde{y}}

\newcommand{\ma}{\ensuremath{\measuredangle}}
\newcommand{\tma}{\ensuremath{\Tilde{\measuredangle}}}
\newcommand{\diff}{\ensuremath{\mathrm{d}}}

\makeatletter
\let\@fnsymbol\@arabic
\makeatother

\allowdisplaybreaks

\title{Total curvature and length estimates for timelike curves in Lorentzian length spaces}
\author{Darius Er\"os\footnotemark[1], \,\,\,Felix Rott\footnotemark[2], \,\,\,Zhe-Feng Xu\footnotemark[2]\,\,\textsuperscript{,}\footnotemark[3]}
\date{\today}

\begin{document}

\maketitle

\footnotetext[1]{Department of Mathematics, University of Vienna, Oskar-Morgenstern-Platz 1, 1090 Vienna, Austria, darius.eroes@univie.ac.at}
\footnotetext[2]{SISSA, 34136, Trieste, Italy, frott@sissa.it, zxu@sissa.it} 
\footnotetext[3]{School of Mathematical Sciences, University of Science and Technology of China,  230026, Hefei, China, xzf1998@mail.ustc.edu.cn}

\begin{abstract}
We introduce and study a synthetic notion of timelike total curvature for curves in Lorentzian length spaces with upper curvature bounds. 
In particular, we prove that our notion agrees with its smooth counterpart, and we show that timelike curves of finite total curvature are rectifiable. 
As the main application, we provide a sharp lower bound for the length of timelike curves solely in terms of the time separation between their endpoints and their total curvature. 

\bigskip

\medskip

\end{abstract}
\tableofcontents

\section{Introduction}
Lorentzian geometry provides the mathematical framework for the study of causality, spacetime structure, and variational problems in general relativity. 
In the smooth setting, this heavily relies on the differentiability of the metric, and involves second-order methods for the study of geodesics and curvature.
Many natural geometric and physical situations, however, lie beyond the smooth category, which motivated the development of synthetic frameworks capable of capturing causality and curvature in low-regularity settings.
The theory of Lorentzian length spaces \cite{KS18} provides such a framework, extending the philosophy of metric geometry to the Lorentzian world and allowing one to study causality theory and curvature bounds without assuming a smooth manifold structure. 
In this work, we propose a synthetic notion of timelike total curvature and use it to show that timelike curves of prescribed total curvature between two fixed points cannot be arbitrarily short in $\ell$-length. 
In fact, we derive a sharp length estimate whose extremal curves are shown to be $\ell$-length-preserving images of geodesic bi-segments whose interior angle coincides with the prescribed total curvature.

In positive signature, the total curvature captures the turning behaviour of a curve and records the total deviation from being geodesic. 
This concept has been studied extensively, both in smooth and non-smooth settings \cite{Mil50, Gra89, ABG2010, BFG15, CS24}.
Most immediately relevant for us are the results established in \cite{ML2003, CM16}, which give upper bounds on the lengths of curves in CAT($k$) spaces based merely on bounds on total curvature and chord length.
The notion of total curvature in CAT($k$) spaces was originally introduced in \cite{AB98}, see also \cite{LMM2010}, and is based on the idea of summing interior angles of approximating inscribed polygons to yield a natural notion of curvature.

The Lorentzian analogue is considerably subtler, as geometric arguments must additionally take into account the \emph{causal structure} of the space.
Furthermore, the time separation function $\ell$, which determines the geometry of the space, satisfies a \emph{reverse} triangle inequality (along causal chains), and geodesics \emph{maximise} their length as opposed to minimising it. 
Therefore, our result will give a \emph{lower} bound on lengths of timelike curves.
Similarly, we will show that bounded total curvature prohibits (local) collapse to \emph{null} length, i.e., $\ell$-rectifiability.

The development of hyperbolic angles \cite{BS23} allows us to give a definition of total curvature for timelike curves in purely synthetic terms, akin to the one used in the study of CAT($k$) spaces.
A crucial ingredient in the proof of the length estimate in \cite{ML2003} is the Majorisation theorem of Reshetnyak \cite{Res68}, which has recently been adapted to the Lorentzian setting \cite{BR25+}. 
For spacetimes, the notion of total curvature has been considered, e.g., in \cite{FGL01, BCO06, CFO07}, but, to the best of our knowledge, it has not previously been developed in the non-smooth setting. 

Let us state our main theorem.
\begin{theorem}[Lower bound on the length of timelike curves]
Let $X$ be a strongly causal and regular \LpLS and let $K \in \mathbb{R}$. 
Let $\gamma$ be a timelike curve in a $(\leq K)$-comparison neighbourhood. 
Let $x$ and $y$ be the endpoints of $\gamma$. 
Let $r=\ell(x,y)<D_K$ be the time separation of the endpoints and let $\kappa=\TC(\gamma)\in[0,+\infty)$ be the total curvature of $\gamma$. 
Then
\begin{equation}
\label{eq: main thm eq intro}
L(\gamma)\ge \mathcal{L}(K,r,\kappa),
\end{equation}
where
\begin{equation*}
\mathcal{L}(K,r,\kappa):=
\begin{cases}
\dfrac{2}{\sqrt{K}}\arcsin\!\left(\dfrac{\sin(\sqrt{K}r/2)}{\cosh(\kappa/2)}\right), & \mathrm{if}\,\,K>0,\\[1.2ex]
\dfrac{r}{\cosh(\kappa/2)}, & \mathrm{if}\,\,K=0,\\[1.2ex]
\dfrac{2}{\sqrt{-K}}\operatorname{arsinh}\!\left(\dfrac{\sinh(\sqrt{-K}r/2)}{\cosh(\kappa/2)}\right), & \mathrm{if}\,\,K<0.
\end{cases}
\end{equation*} 
If $\gamma$ is a timelike poly-segment and attains equality in \eqref{eq: main thm eq intro}, then $\gamma$ is a timelike geodesic in the case $\kappa=0$, and is the $\ell$-length-preserving image of an isosceles timelike bi-segment in $\lm K$ with total curvature $\kappa$ if $\kappa>0$.
\end{theorem}

\section{Preliminaries}\label{subsection 2.1}
We first recall a few basic notions from non-smooth Lorentzian geometry. For further details, we refer to \cite{KS18, BS23, BKR24}.

\begin{definition}[\LpLSn]
\label{def: LpLS}
A \emph{Lorentzian pre-length space} is a quintuple $(X,\mathsf d,\ll,\leq,\ell)$ such that $(X,\mathsf d)$ is a metric space, $\leq$ is a reflexive and transitive relation, $\ll$ is a transitive relation contained in $\leq$, and the time separation function $\ell\colon X\times X\to[0,+\infty]$ satisfies $\ell(x,z)\geq \ell(x,y)+\ell(y,z)$ whenever $x\leq y\leq z$, and $\ell(x,y)>0$ if and only if $x\ll y$.
\end{definition}

Here, the only interplay between the auxiliary metric $\mathsf d$ and the time separation $\ell$ is topological, and our constructions do not rely on this exact axiomatisation. Other proposed axiomatisations of non-smooth Lorentzian spaces would work equally well.

Next, we recall the definition of the length and the notion of rectifiability for a future-directed causal curve, i.e., for a curve $\gamma\colon [a,b] \to X$ such that $\gamma(s) \leq \gamma(t)$ whenever $s < t$. Timelike and past-directed curves are defined analogously.
\begin{definition}[{$\ell$-length, cf.~\cite[Definition~2.24]{KS18}}]\label{length}
Let $\gamma \colon [a,b] \to X$ be a future-directed causal curve. Then we define its \emph{$\ell$-length} (or just \emph{length}) by
\begin{equation}\label{definition of length}
L(\gamma)
:= \inf \left\{
\sum_{i=0}^{N-1} \ell\bigl(\gamma(t_i),\gamma(t_{i+1})\bigr)
\,:\,
N \in \mathbb{N},\;
a=t_0<t_1<\cdots<t_N=b
\right\}.
\end{equation}
\end{definition}

The definition of $\ell$-length, along with many subsequent ones, evidently applies equally well to past-directed curves. 
Unless stated otherwise, we shall therefore take all causal curves to be future-directed.

\begin{definition}[{Rectifiability, cf.~\cite[Definition~2.29]{KS18}}]
\label{rectifiable of causal curve}
A causal curve $\gamma\colon [a,b]\to X$ is called \emph{rectifiable} if, for all $a\le t_1<t_2\le b$, it holds that
\begin{equation*}
L\bigl(\gamma|_{[t_1,t_2]}\bigr)>0. 
\end{equation*}
\end{definition}

This notion of rectifiability should be understood as dual to the one used in metric geometry, where refining a partition can only increase length, and rectifiability thereby provides a control on the oscillations of the curve. In the Lorentzian setting, in which the time separation satisfies the reverse triangle inequality, and refining can only decrease the length, the degenerate behaviour that is excluded by this notion of rectifiability is (local) collapse to null length.

Since some of our results are formulated for (smooth) spacetimes and the word \emph{geodesic} is slightly ambiguous, we try to emphasise the distinction between geodesics (in the smooth sense) and \emph{maximising geodesics} ($\ell$-maximisers). 
For a maximising geodesic from $x$ to $y$, we also use the notation $[x,y]$, and we may also refer to such a curve as a (geodesic) segment. 

The following two properties ensure that the various notions of curvature bounds that we introduce below are, in fact, equivalent.

\begin{definition}[Causal properties]
\label{def: causal properties}
Let $X$ be a \LpLSn. 
\begin{enumerate}
\item $X$ is called \emph{strongly causal} if the set of all timelike diamonds forms a subbasis for the topology. 
\item $X$ is called \emph{regular} if maximising geodesics have a causal character. 
That is, if $\ell(x,y) > 0$, then any geodesic between $x$ and $y$ is timelike (and does not contain a null segment). 
\end{enumerate}
\end{definition}

For the sake of completeness, we restate the definition of triangle comparison here, and refer the reader to \cite{BKR24, BR25+} for further equivalent notions.
By $\lm{K}$ we denote the Lorentzian model space of constant curvature $K$. 
The constant $D_K$ appearing below is given by $\pi/\sqrt{K}$ if $K > 0$, and by $+ \infty$ if $K \leq 0$. 
We say a triangle $\Delta(x,y,z)$ \emph{satisfies size bounds} if $\ell(x,z) < D_K$, in which case there exists a unique comparison triangle in $\lm{K}$, cf.\ \cite[Lemma 2.1]{AB08}. 
Unless explicitly stated otherwise, we assume all triangles satisfy size bounds.

\begin{definition}[Triangle comparison]
\label{def: triangle comparison}
Let $X$ be a \LpLSn. An open subset $U$ is called a $(\leq K)$-comparison neighbourhood in the sense of triangle comparison if:
\begin{enumerate}[label=(\roman*)]
\item $\ell$ is continuous on $(U\times U) \cap \ell^{-1}([0,D_K))$, and this set is open; 
\item For all $x \ll y$ in $U$ with $\ell(x,y)<D_K$, there exists a connecting maximising geodesic inside $U$;
\item Let $\Delta (x,y,z)$ be a timelike triangle in $U$, with $p,q$ two points on the sides of $\Delta (x,y,z)$. Let $\Delta(\bar{x}, \bar{y}, \bar{z})$ be a comparison triangle in $\lm{K}$ for $\Delta (x,y,z)$ and $\bar{p},\bar{q}$ comparison points for $p$ and $q$, respectively. Then 
\begin{equation*}
\label{eq: timelike triangle comparison inequality}
\ell(p,q) \geq \ell(\bar{p}, \bar{q}) \, .
\end{equation*}
\end{enumerate}
$X$ is said to have timelike curvature bounded from above by $K$ if it is covered by $(\leq K)$-comparison neighbourhoods. 
$X$ is said to have timelike curvature globally bounded from above if $X$ itself is a comparison neighbourhood. 
\end{definition}

Next, we recall the definition of angles in \LpLSsn, initially introduced in \cite{BS23}.

\begin{definition}[Angles]
Let $X$ be a \LpLS and let $\alpha,\beta\colon [0,\eps)\rightarrow X$ be two timelike curves (where we permit one or both of the curves to be past-directed) with $x:=\alpha(0)=\beta(0)$.
Define $D$ as the set of points $(s,t) \in [0,\eps)^2$ such that $\alpha(s)$ and $\beta(t)$ are timelike related and the timelike triangle formed by $x, \alpha(s)$ and $\beta(t)$ satisfies size bounds. 
Then we define the \emph{angle} 
\begin{equation}
\label{eq: def angle}
\ma_x(\alpha,\beta)=\limsup_{\substack{(s,t)\in D \\ s,t\to 0}}\tilde{\ma}_x^{K}\big(\alpha(s),\beta(t)\big),
\end{equation}
where $\tilde{\ma}_x^{K}(\alpha(s),\beta(t))$ is the \emph{$K$-comparison angle}, i.e., the respective (hyperbolic) angle in the comparison triangle in $\lm{K}$. 
\end{definition}

Recall that \eqref{eq: def angle} is independent of the choice of $K$, cf.\ \cite[Proposition 2.14]{BS23}, and that if $X$ is regular and satisfies the assumptions of Definition \ref{def: triangle comparison}, the $\limsup$ in \eqref{eq: def angle} is a limit, cf.\ \cite[Lemma 4.10]{BS23}. 
We may also drop the decoration $K$ when dealing with comparison angles.
Given three timelike related points $x,y,z$ in a \LpLS $X$, we denote by $\ma_y(x,z)$ the angle at $y$ between $[x,y]$ and $[y,z]$. 
In our setting, geodesics are either unique, cf.\ \cite[Theorem 4.7]{BNR25}, or the choice of geodesic is implicitly clear, so this notation is justified. 
Further recall the \emph{sign} of the angle, here denoted by $\epsilon$, which is $-1$ if the angle is measured at a time-endpoint, and $+1$ otherwise. 
We then introduce the \emph{signed angle} as $\ma_x^S(y,z):=\epsilon\ma_x(y,z)$. 

Let us mention two more notions of curvature bounds that are important for us, cf., respectively, \cite[Definition 3.11]{BKR24} and \cite[Lemma 4.8]{BR25+}. 

\begin{definition}[Angle comparison and four-point condition]
\label{def: angle and 4pt}
Assume a strongly causal and regular \LpLS $X$ can be covered by neighbourhoods as in Definition \ref{def: triangle comparison}. 
Then there also exists a cover of neighbourhoods of $X$ where (iii) can be replaced with either of the following:
\begin{enumerate}[label=(\roman*)]
\item For any timelike triangle $\Delta(x,y,z)$ in $U$ we have 
\begin{equation*}
\ma^S_x(y,z) \geq \tma_x^{K,S}(y,z), \qquad \ma^S_y(x,z) \geq \tma_y^{K,S}(x,z), \qquad \ma^S_z(x,y) \geq \tma_z^{K,S}(x,y).
\end{equation*}
\item For any $x_1 \ll x_2 \ll x_3 \ll x_4$ in $U$, we have 
\begin{equation*}
\tilde{\ma}_{x_1}(x_2,x_3)\leq\tilde{\ma}_{x_1}(x_2,x_4)+\tilde{\ma}_{x_1}(x_4,x_3), \qquad \tilde{\ma}_{x_2}(x_1,x_4)\leq\tilde{\ma}_{x_2}(x_1,x_3)+\tilde{\ma}_{x_2}(x_3,x_4).
\end{equation*}
\end{enumerate}
In particular, if $X$ can be chosen as its own cover in Definition \ref{def: triangle comparison}, then the same is true here.
\end{definition}

Let us also state the Majorisation Theorem, one of our main tools, which was recently introduced in \cite{BR25+}. 
In fact, we shall only make use of the following very specialised case of said theorem, which easily follows by the nature of the proof of \cite[Theorem 3.9]{BR25+}.

\begin{theorem}[Majorisation Theorem, special case]
\label{thm: majorisation}
Let $X$ be a strongly causal and regular \LpLS with curvature bounded above by $K \in \R$ and let $U$ be a ($\leq K$)-comparison neighbourhood. 
Let $O \ll z$ with $\ell(O,z)<D_K$ in $U$ and let $\alpha$ be a poly-segment from $O$ to $z$. 
Then there exists a timelike poly-segment $\tilde \alpha$ in $\lm K$ between points $\tilde O$ and $\tilde z$ such that $\ell(O,z)=\ell(\tilde O, \tilde z)$ and $\tilde \alpha$ and $[\tilde O, \tilde z]$ bound a convex region $R$ and a map $f\colon R \to U$ that satisfies $\ell(x,y) \leq \ell(f(x),f(y))$ for all $x,y \in R$ and such that $\tilde \alpha$ and $[\tilde O, \tilde z]$, which form the boundary of $R$, are, respectively, mapped onto $\alpha$ and $[O,z]$ such that the $\ell$-length of the curves is preserved.
\end{theorem}

Next, recall the definition of hyperbolic angles in a spacetime. 
If $\alpha$ and $\beta$ are two timelike curves (not necessarily future-directed) meeting at $p=\alpha(0)=\beta(0)$, then the \emph{hyperbolic angle} at $p$ between $\alpha$ and $\beta$ is 
\begin{equation}
\label{eq: def angle smooth}
\ma_p(\alpha,\beta):=\arcosh{\biggl(\frac{|g_p(\dot\alpha(0),\dot\beta(0))|}{\|\dot\alpha(0)\| \|\dot\beta(0)\|}\biggr)}, 
\end{equation}
where $\|v\|:=\sqrt{g_p(v,v)}$ for $v \in T_pM$. 
We will again use the notation $\ma_p(x,y)$ if the choice of curves is clear or irrelevant. 
If $(M,g)$ is strongly causal, then by \cite[Proposition 2.13 \& Corollary 4.7]{BS23}, the two notions of angle \eqref{eq: def angle} and \eqref{eq: def angle smooth} agree. 

Next, we recall the Law of Cosines.\footnote{Notice that the cases depending on the sign of $K$ are switched when comparing with \cite{BS23}. 
This is because we use the convention $(+,-,\ldots,-)$ for the signature of the metric $g$, as we will state below. 
See \cite[Remark 2.3]{RXZ26+} and \cite[Section 2.1]{GRZ26+} for more details.}

\begin{proposition}[{Law of Cosines, cf.~\cite[Lemma~2.4 \& Remark~2.5]{BS23}}]
\label{prop: law of cosines monotonicity}
Let $p,q,r \in \mathbb{L}^2(K)$ form a timelike triangle
(not necessarily in this order). Let $a=\max\{\ell(p,q),\ell(q,p)\}>0,
b=\max\{\ell(q,r),\ell(r,q)\}>0,$ and $c=\max\{\ell(p,r),\ell(r,p)\}.$ Let $\omega=\ma_q(p,r)$ be the hyperbolic angle at $q$ and let $\epsilon$
be its sign. Let $s=\sqrt{|K|}$. Then we have:
\[
\begin{cases}
a^2+b^2=c^2-2ab\epsilon\cosh(\omega), & K=0,\\[0.4em]
\cos(sc)=\cos(sa)\cos(sb)-\epsilon\cosh(\omega)\sin(sa)\sin(sb), & K>0,\\[0.4em]
\cosh(sc)=\cosh(sa)\cosh(sb)+\epsilon\cosh(\omega)\sinh(sa)\sinh(sb), & K<0.
\end{cases}
\]
\end{proposition}

In contrast to the metric setting, fixing two side lengths and varying the third, \emph{all} angles are strictly monotonically increasing if the varied side is the longest, and decreasing if it is one of the short sides. 

Finally, we will make use of the local Gauss--Bonnet Formula. 
More precisely, we shall employ a specialised version of this formula, which easily follows from, e.g., \cite[Lemma 3.4]{BBCGRR26+}. 
See also \cite{Jee84}. 

\begin{lemma}[Gauss--Bonnet Formula]
\label{lem: angle sum minkowski}
Let $K \in \R$ and let $\Delta(x,y,z)$ be a timelike triangle in $\lm K$. 
Then 
\begin{equation*}
\ma_x^S(y,z) + \ma_y^S(x,z) + \ma_z^S(x,y) = -K \vol\big(\Delta(x,y,z)\big).
\end{equation*}
Here, by slight abuse of notation, $\vol(\Delta(x,y,z))$ denotes the area of the surface spanned by the triangle. 
In particular, in Minkowski space, we have $\ma_x^S(y,z) + \ma_y^S(x,z) + \ma_z^S(x,y) = 0$. 
\end{lemma}

\section{Total curvature of timelike curves}
\label{compatibility}

In this section, we first recall the smooth definition of the total curvature of (timelike) curves, then propose a corresponding non-smooth analogue, and show their equivalence in strongly causal spacetimes. 
We also give an alternative formulation of total curvature for non-positively curved spaces. 

Let $(M,g)$ be a spacetime with signature $(+, -, \dots, -)$ and let $\gamma\colon  [a,b]\rightarrow M$ be a timelike $C^2$ curve parametrised by proper time. Then $T(s):=\dot\gamma(s)$ satisfies
\begin{equation*}
g\big(T(s),T(s)\big)\equiv 1.
\end{equation*}
Differentiating the above equation along $\gamma$, we obtain $ g(\nabla_TT(s),T(s))=0$, which implies that $\nabla_T T(s)$ is spacelike, i.e., $ g(\nabla_TT(s),\nabla_TT(s))\le 0$. 
Denote by $$\kappa(s):=\sqrt{-g\big(\nabla_TT(s),\nabla_TT(s)\big)}\in [0,+\infty)$$ the curvature of $\gamma$ at the point $\gamma(s)$. The \emph{(smooth) total curvature} of $\gamma$ is then computed by
\begin{equation}
\label{deftc2}
\int_a^b \kappa(s) \,\mathrm{d}s.
\end{equation}

\begin{definition}[Timelike poly-segments]
\label{def: poly-segments}
A \emph{timelike poly-segment} in a Lorentzian pre-length space $X$ is a timelike curve $\sigma\colon [a,b]\to X$ together with a partition $a=t_0<t_1<\cdots<t_n=b$ such that for each $i=1,\dots,n$, 
the respective restriction $\sigma|_{[t_{i-1},t_i]}$ is a (non-constant) maximising geodesic from $p_{i-1}$ to $p_i$, where we have set $p_i:=\sigma(t_i)$. 

If $n=2$, then $\sigma$ is called a \emph{timelike bi-segment}. 
\end{definition}

Since we never deal with poly-segments that contain null segments, we will occasionally drop the adjective `timelike' to increase readability. 

\begin{definition}[Total rotation and curvature]
Let $X$ be a \LpLS and let $\sigma$ be a timelike poly-segment as in Definition \ref{def: poly-segments}. 
The \emph{total rotation} $\kappa^*(\sigma)\in[0,+\infty)$ of $\sigma$ is defined as the sum of the angles at the interior vertices:
\begin{equation*}
\kappa^*(\sigma):=\sum_{i=1}^{n-1}\ma_{p_i}(p_{i-1},p_{i+1}).
\end{equation*}

Let $\gamma\colon [a,b]\to X$ be a timelike curve. 
A timelike poly-segment $\sigma$ is said to be \emph{inscribed} in $\gamma$ if there exists a partition $a=t_0<t_1<\cdots<t_n=b$ and a parametrisation of $\sigma$ on $[a,b]$ such that $\sigma(t_i)=\gamma(t_i)$ for $0\leq i\leq n$. 
Denote by
\begin{equation*}
|\sigma|:=\max_{0\leq i\leq n-1}(t_{i+1}-t_i)
\end{equation*}
the \emph{mesh-size} of $\sigma$. 
We define the \emph{(synthetic) total curvature} of $\gamma$ by
\begin{equation}
\label{deftc1}
\mathrm{TC}(\gamma):=\limsup_{|\sigma|\to 0}\kappa^*(\sigma)=\lim_{\eps\to 0^+}\sup_{\sigma\in\Sigma_\eps(\gamma)}\kappa^*(\sigma),
\end{equation}
where $\Sigma_\eps(\gamma)$ denotes the collection of timelike poly-segments $\sigma$ inscribed in $\gamma$ with $|\sigma|<\eps$. 
\end{definition}

We now show that, for timelike curves in strongly causal spacetimes, the formulas \eqref{deftc1} and \eqref{deftc2} agree.\footnote{Strong causality is imposed to guarantee that \eqref{deftc1} is independent of the notion of angle, \eqref{eq: def angle} or \eqref{eq: def angle smooth}, we consider. 
If we instead defined \eqref{deftc1} with \eqref{eq: def angle smooth} in mind, the compatibility follows for any spacetime.} 
The proof is similar in spirit to the corresponding arguments in the Riemannian setting, see \cite[Theorem 3.4]{LMM2010}.
\begin{lemma}[Derivative of the hyperbolic angle]
\label{lem}
Let $(M,g)$ be a spacetime, and let $\gamma \colon [a,b] \to M$ be a timelike $C^2$ curve parametrised by proper time. 
Denote by $T(s)$ the tangent vector to $\gamma$ at $\gamma(s)$. 
For fixed $s \in (a,b)$, define
\begin{equation*}
A(t):=T(s)+\frac12 t\nabla_TT(s)+\rho(t),
\end{equation*}
where
\begin{equation*}
\rho(t):=
\begin{cases}
\frac{\Gamma(t)}{t}-T(s)-\frac{1}{2}t\nabla_TT(s), & t\neq 0,\\
0, & t=0,
\end{cases} 
\qquad \Gamma(t):=\exp_{\gamma(s)}^{-1}\bigl(\gamma(s+t)\bigr).    
\end{equation*}
Then there exists $\delta > 0$ such that if $|t|<\delta$, then $A(t)$ is future-directed timelike. 
For any $-\delta<r_1<0<r_2<\delta$, define
\begin{equation*}
\theta_{\gamma(s)}(r_1,r_2):=\arcosh\left(g_{\gamma(s)}\left(\frac{A(r_1)}{\sqrt{g_{\gamma(s)}(A(r_1),A(r_1))}},\frac{A(r_2)}{\sqrt{g_{\gamma(s)}(A(r_2),A(r_2))}}\right)\right), 
\end{equation*}
and $\theta_{\gamma(s)}(0,0)=0$.
Then 
\begin{equation*}
\theta_{\gamma(s)}(r_1,r_2) = \frac{\kappa(s)}{2}(r_2-r_1) + o(|r_1|+|r_2|).
\end{equation*}
In particular, 
\begin{equation}
\label{eq: partial derivatives angle}
\frac{\partial \theta_{\gamma(s)}}{\partial r_1}(0,0)=-\frac12 \kappa(s),
\qquad
\frac{\partial \theta_{\gamma(s)}}{\partial r_2}(0,0)=\frac12 \kappa(s).
\end{equation}
\end{lemma}

\begin{proof}
Fix $s\in (a,b)$. 
Note that for $t \neq 0$, $A(t)=\frac1t \exp_{\gamma(s)}^{-1}(\gamma(s+t))$ is the initial velocity vector of the geodesic segment $[\gamma(s),\gamma(s+t)]$. 
Since $\gamma$ is not necessarily a geodesic, however, $A(t)$ need not be a unit vector.
Similarly, $\theta(r_1,r_2)$ denotes the hyperbolic angle at $\gamma(s)$ between the segments $[\gamma(s+r_1),\gamma(s)]$ and $[\gamma(s),\gamma(s+r_2)]$. 
By a standard computation in normal coordinates in a small neighbourhood of $\gamma(s)$, we know that $\Gamma(0)=0$, $\Gamma'(0)=T(s)$ and $ \Gamma''(0)=\nabla_TT(s)$. 
Combining this with a Taylor expansion, we have
\begin{equation*}
\Gamma(t)=tT(s)+\frac12 t^2\nabla_TT(s)+o(t^2) \quad \mathrm{around}\,\,0,
\end{equation*}
and hence $\rho(t)=o(t)$. By definition of $A(t)$, using that $g_{\gamma(s)}(T,\nabla_T T)=0$, we thus have $g_{\gamma(s)}(A(t), A(t)) = 1 + o(t)$ and so
\begin{equation}
\label{eq:expansionSqrA}
\frac{1}{\sqrt{g_{\gamma(s)}(A(t),A(t))}} = 1 + o(t).
\end{equation}
From now on, we omit the parameter $s$ in $T(s)$, $\nabla_T T(s)$ and $\kappa(s)$, and omit the base point $\gamma(s)$ in $g_{\gamma(s)}$ and $\theta_{\gamma(s)}$.
Setting
\begin{equation*}
\label{eq: definitition U(t)}
U(t):=\frac{A(t)}{\sqrt{g(A(t),A(t))}},
\end{equation*}
we get that
\begin{equation}
\label{eq:expansionU}
U(t) = T + \tfrac{1}{2}t\nabla_T T + o(t)
\end{equation}
by combining \eqref{eq:expansionSqrA} and using the definition of $A(t)$. Hence, there exists $\delta >0$ such that $U(t)$ is unit timelike for all $t \in (-\delta, \delta)$. Setting $V(t) = U(t) - T$, this gives
\begin{equation*}
1 = g\big(U(t), U(t)\big) = 1 + 2 g\big(T, V(t)\big) + g\big(V(t), V(t)\big),
\end{equation*}
and so $g(T,V(t)) = - \tfrac{1}{2}g(V(t), V(t))$. Now for $r_1, r_2 \in \R$ such that $-\delta < r_1 < 0 < r_2 < \delta$, we have
\begin{align*}
g\big(U(r_1), U(r_2)\big) &= 1 + g\big(T,V(r_1)\big) + g\big(T, V(r_2)\big) + g\big(V(r_1),V(r_2)\big)\\
&= 1 - \tfrac{1}{2}g\big(V(r_1) - V(r_2), V(r_1) - V(r_2)\big).
\end{align*}
Using the expansion \eqref{eq:expansionU}, we get
\begin{align*}
g\big(V(r_1) - V(r_2), V(r_1) - V(r_2)\big) &= \frac{(r_1-r_2)^2}{4}g(\nabla_T T, \nabla_T T) + o\bigl((|r_1| + |r_2|)^2\bigr) \\
&= -\frac{\kappa^2}{4}(r_1-r_2)^2 + o\bigl((|r_1| + |r_2|)^2\bigr)
\end{align*}
and thus,
\begin{equation*}
\cosh \bigl(\theta(r_1,r_2)\bigr) =g\big(U(r_1), U(r_2)\big) = 1 +\frac{\kappa^2}{8}(r_1-r_2)^2 + o\bigl((|r_1| + |r_2|)^2\bigr).
\end{equation*}
With the Taylor expansion $\cosh{x} = 1 +\tfrac{1}{2}x^2 + o(x^2)$, this implies
\begin{equation*}
\theta(r_1,r_2) = \frac{\kappa}{2}(r_2-r_1) + o(|r_1|+|r_2|), 
\end{equation*}
from which \eqref{eq: partial derivatives angle} follows directly. 
\end{proof}

\begin{proposition}[Notions of total curvature agree]
\label{propcoin}
Let $(M,g)$ be a strongly causal spacetime and let $\gamma\colon  [a,b]\rightarrow M$ be a timelike $C^2$ curve parametrised by proper time. 
Then
\begin{equation*}
\mathrm {TC}(\gamma)\stackrel{\textnormal{def}}{=}\lim_{\eps\to 0^+}\sup_{\sigma\in\Sigma_\eps(\gamma)}\kappa^*(\sigma)=\int_a^b \kappa(s) \,\mathrm{d}s.
\end{equation*}
\end{proposition}

\begin{proof}
Let $\sigma$ be a timelike poly-segment inscribed in $\gamma$ with sufficiently small mesh-size $|\sigma|$. Let $a=t_0<\cdots<t_n=b$ be a partition such that $p_i:=\sigma(t_i)=\gamma(t_i)$. For each interior vertex $p_i \,(1\le i\le n-1)$, set
\begin{equation*}
r^i_1:=t_{i-1}-t_i<0,\qquad r^i_2:=t_{i+1}-t_i>0.
\end{equation*}
Let us call the normalised timelike vectors $U(r^i_1)$ and $U(r^i_2)$ used in Lemma \ref{lem} $v_i$ and $w_i$, respectively. 
Then
\begin{equation*}
\ma_{p_i}(v_i, w_i)=\theta_{\sigma(t_i)}(r^i_1,r^i_2)=\theta_{\sigma(t_i)}(t_{i-1}-t_i,\ t_{i+1}-t_i).
\end{equation*}
By Lemma \ref{lem} and the first-order Taylor expansion of $\theta_{\sigma(t_i)}(r^i_1,r^i_2)$ around $(0,0)$,
\begin{align*}
\ma_{p_i}(v_i, w_i)
&=\frac{\partial\theta_{\sigma(t_i)}}{\partial r_1}(0,0)\,(t_{i-1}-t_i)
+\frac{\partial\theta_{\sigma(t_i)}}{\partial r_2}(0,0)\,(t_{i+1}-t_i)
+R(r^i_1,r^i_2)\\
&=\frac12\,\kappa(t_i)(t_i-t_{i-1})
+\frac12\,\kappa(t_i)(t_{i+1}-t_i)
+R(r^i_1,r^i_2),
\end{align*}
where $R(r^i_1,r^i_2) = o(|r_1^{i}|+|r_2^{i}|)$ denotes the remainder term.
By Lemma \ref{lem}, for each fixed $s \in [a,b]$ and for all $r_1, r_2 \in \R$ such that $-\delta < r_1 < 0 < r_2 < \delta$, we have 
\begin{equation*}
\theta_{\gamma(s)}(r_1,r_2) = \frac{\kappa(s)}{2}(r_2-r_1) + o(|r_1| + |r_2|).
\end{equation*}
Now, by compactness of $[a,b]$ and by continuity of $\theta_{\gamma(s)}$ and $\kappa(s)$ in $s$, this estimate is uniform in $s$ in the sense that, for each $\eps >0$, there exists some $\delta >0$ such that for all $s \in [a,b]$, it holds that
\begin{equation*}
\biggl|\theta_{\gamma(s)}(r_1,r_2) - \frac{\kappa(s)}{2}(r_2-r_1)\biggr| \leq \eps (|r_1|+|r_2|),
\end{equation*}
provided $|r_1| + |r_2| < \delta$. Hence, in our situation, 
for any $\eps>0$, there exists $\delta>0$ such that if $|r^i_1|+|r^i_2|<\delta$, then
\begin{equation*}
\big|R(r^i_1,r^i_2)\big|\leq \eps\big(|r_1^{i}|+|r_2^{i}|\big)
=\eps\bigl((t_i-t_{i-1})+(t_{i+1}-t_i)\bigr)
=\eps(t_{i+1}-t_{i-1}).
\end{equation*}
Thus, for $|\sigma|$ sufficiently small, we have 
\begin{equation*}
\sum_{i=1}^{n-1}\big|R(r^i_1,r^i_2)\big|
\leq \eps\sum_{i=1}^{n-1}(t_{i+1}-t_{i-1})
\leq 2\eps(b-a).
\end{equation*}
Therefore,
\begin{equation*}
\sum_{i=1}^{n-1}\ma_{p_i}(v_i, w_i)
=
\frac{1}{2}\sum_{i=1}^{n-1}\kappa(t_i)(t_i-t_{i-1})
+\frac{1}{2}\sum_{i=1}^{n-1}\kappa(t_i)(t_{i+1}-t_i)
+\sum_{i=1}^{n-1}R(r^i_1,r^i_2).
\end{equation*}
Letting $|\sigma|\to 0$ and $\eps\to 0$, we have
\begin{equation*}
\mathrm {TC}(\gamma)=\limsup_{|\sigma|\rightarrow 0} \sum_{i=1}^{n-1}\ma_{p_i}(v_i, w_i)
=\lim_{|\sigma|\rightarrow 0} \sum_{i=1}^{n-1}\ma_{p_i}(v_i, w_i)
=\int_a^b \kappa(s)\diff s. \qedhere
\end{equation*}
\end{proof}

Next, we give an alternative formulation of the total curvature on non-positively curved spaces. 
Namely, the total curvature of a timelike curve can be expressed as the supremum over all inscribed poly-segments, cf.\ \cite[Theorem 4.5]{LMM2010}. 
This follows easily as a corollary of the following statement. 

\begin{lemma}[Adding a vertex does not decrease total curvature]
\label{lem: TCpoly}
Let $X$ be a strongly causal and regular Lorentzian pre-length space and let $\gamma$ be a timelike poly-segment in a $(\leq 0)$-comparison neighbourhood. 
Denote the vertices of $\gamma$ by $p_0,\dots,p_n$. 
Let $\gamma'$ be a timelike poly-segment obtained from $\gamma$ by adjoining one new vertex $\bar{p}$, where the vertices of $\gamma'$ are given by
\[
\begin{cases}
\bar{p},p_0,\dots,p_n, & \mathrm{if }\,\, \bar{p}\in I^-(p_0),\\
p_0,\dots,p_i,\bar{p},p_{i+1},\dots,p_n, & \mathrm{if }\,\, \bar{p}\in I^+(p_i)\cap I^-(p_{i+1})\, \,\mathrm{ for\, some }\, \,i\in\{0,\dots,n-1\},\\
p_0,\dots,p_n,\bar{p}, & \mathrm{if }\, \,\bar{p}\in I^+(p_n).
\end{cases}
\]
Then
\begin{equation*}
\kappa^*(\gamma')\geq \kappa^*(\gamma).
\end{equation*}
\end{lemma}

\begin{proof}  
If $n=1$, the statement is trivial. 
If $n \geq 2$, the first and third cases are immediate. 
Indeed, in these cases, the total rotation of $\gamma'$ is the total rotation of $\gamma$ plus the additional angle contributed by the new vertex. 

In the remaining case $p_i \ll \bar{p} \ll p_{i+1}$, we first suppose $1 \leq i \leq n-2$ and set
\begin{equation*}
\begin{aligned}
&\alpha_1:=\ma_{p_i}(p_{i-1},p_{i+1}),
&&\,
&&\alpha_2:=\ma_{p_{i+1}}(p_i,p_{i+2}), \\
&\beta_1:=\ma_{p_i}(p_{i-1},\bar{p}),
&&\beta_2:=\ma_{\bar{p}}(p_i,p_{i+1}),
&&\beta_3:=\ma_{p_{i+1}}(\bar{p},p_{i+2}),
\end{aligned}
\end{equation*}
cf.\ Figure \ref{fig: add vertex no curvature}. 
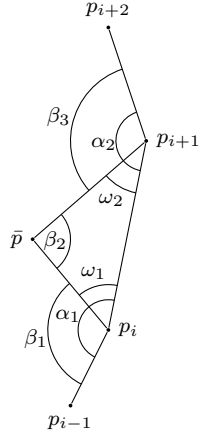
\begin{figure}
\begin{center}
\begin{tikzpicture}
\begin{scriptsize}
\coordinate [circle, fill=black, inner sep=0.5pt, label=270: {$p_{i-1}$}] (i-1) at (0,0);
\coordinate [circle, fill=black, inner sep=0.5pt, label=0: {$p_{i}$}] (i) at (0.5,1);
\coordinate [circle, fill=black, inner sep=0.5pt, label=180: {$\bar p$}] (bp) at (-.5,2.2);
\coordinate [circle, fill=black, inner sep=0.5pt, label=0: {$p_{i+1}$}] (i+1) at (1,3.5);
\coordinate [circle, fill=black, inner sep=0.5pt, label=90: {$p_{i+2}$}] (i+2) at (0.5,5);

\pic [draw, "${\beta_1}$", angle radius=8mm, angle eccentricity=1.2] {angle = bp--i--i-1};
\pic [draw, "${\beta_2}$", angle radius=5mm, angle eccentricity=0.6] {angle = i--bp--i+1};
\pic [draw, "${\beta_3}$", angle radius=10mm, angle eccentricity=1.2] {angle = i+2--i+1--bp};

\pic [draw, "${\omega_1}$", angle radius=6mm, angle eccentricity=1.3] {angle = i+1--i--bp};
\pic [draw, "${\omega_2}$", angle radius=7mm, angle eccentricity=1.3] {angle = bp--i+1--i};

\pic [draw, "${\alpha_1}$", angle radius=4mm, angle eccentricity=1.4] {angle = i+1--i--i-1};
\pic [draw, "${\alpha_2}$", angle radius=4mm, angle eccentricity=1.4] {angle = i+2--i+1--i};
\end{scriptsize}
\draw (i-1) -- (i) -- (i+1) -- (i+2);
\draw (i) -- (bp) -- (i+1);
\end{tikzpicture}
\end{center}
\caption{Adding a vertex does not decrease the curvature of the poly-segment.}
\label{fig: add vertex no curvature}
\end{figure}
Then
\[
\kappa^*(\gamma')-\kappa^*(\gamma)=\beta_1+\beta_2+\beta_3-\alpha_1-\alpha_2,
\]
and we now show that this quantity is nonnegative. Set
\[
\omega_1=\ma_{p_i}(\bar{p},p_{i+1}),\quad \omega_2=\ma_{p_{i+1}}(p_i,\bar{p}).
\]
By Definition \ref{def: angle and 4pt}(ii) applied to the timelike four-point configuration $p_{i-1} \ll p_i \ll \bp \ll p_{i+1}$, we have $\bar\alpha_1 \leq \bar\beta_1 + \bar\omega_1$, where these denote the respective comparison angles. 
Choosing appropriate sequences of parameters along the respective geodesics, this inequality is inherited in the limit, i.e., we have $\alpha_1\le \beta_1+\omega_1$. 
Similarly, we obtain $\alpha_2\le \beta_3+\omega_2$. 

Consider now the triangle $\Delta(p_i,\bar{p},p_{i+1})$. 
It has interior angles $\omega_1$, $\beta_2$, and $\omega_2$. Denote by $\beta_2^S$ the signed angle and by $\bar\beta_2^S$ the signed comparison angle, and use similar notation for $\omega_1$ and $\omega_2$. 
Then
\begin{equation*}
\label{eq:SumOfAnglesVertex}
\beta_1+\beta_2+\beta_3-\alpha_1-\alpha_2 \ge \beta_2-\omega_1-\omega_2 = \beta_2^S+\omega_1^S+\omega_2^S \ge \bar\beta_2^S+\bar\omega_1^S+\bar\omega_2^S = 0,
\end{equation*}
where the second inequality is due to angle comparison, cf.\ Definition \ref{def: angle and 4pt}(i), and the final equality follows from Lemma \ref{lem: angle sum minkowski}.
In the case $i=0$, i.e., if $p_0 \ll \bar{p} \ll p_1$, we have
\begin{equation*}
\kappa^*(\gamma')-\kappa^*(\gamma)= \beta + \gamma - \alpha,
\end{equation*}
where $\alpha = \ma_{p_1}(p_0,p_2), \beta = \ma_{\bar{p}}(p_0,p_1)$ and $\gamma = \ma_{p_1}(\bar{p},p_2)$. 
Introducing $\omega = \ma_{p_1}(p_0,\bar{p})$ and $\eta = \ma_{p_0}(\bar{p},p_1)$, by Definition \ref{def: angle and 4pt}(ii), we again have $\alpha \leq \omega + \gamma$ and so it remains to show that $\omega \leq \beta$. 
As above,
\begin{equation*}
\beta - \omega - \eta = \beta^S + \omega^S + \eta^S \geq \bar\beta^S + \bar\omega^S + \bar\eta^S = 0,
\end{equation*}
and thus, $\beta \geq \omega + \eta \geq \omega$. 
The case $i = n-1$ is treated similarly.
\end{proof}

\begin{corollary}[Total curvature as supremum]
Let $X$ be a strongly causal and regular \LpLS and let $\gamma$ be a timelike curve in a $(\leq 0)$-comparison neighbourhood. 
Then $ \mathrm {TC}(\gamma)=\sup_{\sigma\in \mathsf{\Pi}}\{ \kappa^*(\sigma)\}$, where $\mathsf{\Pi}$ is the collection of all timelike poly-segments inscribed in $\gamma$.
\end{corollary}

\begin{proof}
From the definition of $\mathrm{TC}(\gamma)$, it follows that there exists a sequence $\{\sigma_i\}_{i\in\mathbb{N}}$ of timelike poly-segments inscribed in $\gamma$, with $|\sigma_i|\to 0$, such that $\mathrm{TC}(\gamma)=\lim_{i\to+\infty} \kappa^*(\sigma_i).$ Certainly, for every $i\in \mathbb{N}$, $\kappa^*(\sigma_i)\leq \sup_{\sigma\in \mathsf{\Pi}} \kappa^*(\sigma),$ which immediately implies $\mathrm{TC}(\gamma)\leq \sup_{\sigma\in \mathsf{\Pi}} \kappa^*(\sigma).$

Conversely, let $P$ be an arbitrary timelike poly-segment inscribed in $\gamma$. 
Denote by $\sigma_i'$ the poly-segment whose set of vertices is the union of those of $P$ and $\sigma_i$. 
Then by Lemma \ref{lem: TCpoly}, $\kappa^*(P) \leq \kappa^*(\sigma_i')$. 
Taking the $\limsup$, we infer $\kappa^*(P) \leq \limsup_{i \rightarrow +\infty} \kappa^*(\sigma_i') \leq \TC(\gamma)$, and the claim follows by the arbitrariness of $P$. 
\end{proof}

We conclude this section by showing that the total rotation and the total curvature for poly-segments agree in spaces with an upper curvature bound.

\begin{proposition}[Total curvature and rotation of poly-segments]
\label{prop: curvatures agree}
Let $X$ be a strongly causal, regular \LpLS and let $\gamma$ be a poly-segment in a $(\leq K)$-comparison neighbourhood with vertices $p_0,\ldots,p_m$. 
Suppose that $\sup_{0\leq i<m} \ell(p_i,p_{i+1}) < D_K$. 
Then $\TC(\gamma)=\kappa^*(\gamma)$. 
\end{proposition}

\begin{proof}
If $m=1$, i.e., if $\gamma$ is a geodesic, the statement follows immediately. 
Clearly, $\gamma$ is a poly-segment inscribed in itself. 
Let $\sigma_n$ be a sequence of poly-segments with $|\sigma_n| \to 0$ such that the set of vertices of $\sigma_n$ contains those of $\gamma$ for each $n$. 
As the angle only depends on the germs of geodesics, we infer $\kappa^*(\gamma) \leq \kappa^*(\sigma_n)$. 
Taking the $\limsup$, we have $\kappa^*(\gamma) \leq \limsup_{n \to \infty}\kappa^*(\sigma_n) \leq \TC(\gamma)$. 

Conversely, given $\delta >0$, let $\sigma$ be any poly-segment inscribed in $\gamma$ with $|\sigma|<\delta$. 
Set $p_i=\gamma(t_i)$. 
By shrinking $\delta$ if necessary, suppose that $\sigma$ has at least three break-points on each segment $[p_i,p_{i+1}]$. 
For each $1 \leq i < m$, let $x_i$ denote the vertex of $\sigma$ associated to the largest parameter not greater than $t_i$, and $y_i$ the vertex of $\sigma$ associated to the smallest parameter larger than $t_i$. 
That is, if $x_i=\gamma(r_i)$ and $y_i=\gamma(s_i)$, we have $r_i \leq t_i < s_i$. 
All other angles between break-points of $\sigma$ are along a single segment, and hence vanish. 
Set $\omega_i=\ma_{p_i}(p_{i-1},p_{i+1})$, $\theta_{i,+}=\ma_{y_i}(x_i,p_{i+1})$, and $\theta_{i,-}=\ma_{x_i}(p_{i-1},y_i)$.
We further denote $\alpha_{i,-}=\ma_{x_i}(p_i,y_i)$ and $\alpha_{i,+}=\ma_{y_i}(x_i,p_i)$. 
Then by \cite[Lemma 4.14]{BKR24}, we have $\theta_{i,\pm} \leq \alpha_{i,\pm}$. 
Now $\alpha_{i,-}, \omega_i$ and $\alpha_{i,+}$ are precisely the angles of the timelike triangle $T_i=\Delta(x_i,p_i,y_i)$. 
Denote its comparison triangle by $\bar T_i$ and the respective comparison angles by $\bar \alpha_{i,-}$, $\bar \omega_i$ and $\bar \alpha_{i,+}$. 
The corresponding signed angles are decorated with an upper index $S$.
Then we get:
\begin{align*}
\kappa^*(\sigma) - \kappa^*(\gamma) 
& = \sum_{i=1}^{m-1} \theta_{i,+} + \,\theta_{i,-} - \omega_i 
\leq  \sum_{i=1}^{m-1} \alpha_{i,+} + \,\alpha_{i,-} - \omega_i 
= \sum_{i=1}^{m-1} -\alpha_{i,+}^{S} - \,\alpha_{i,-}^{S} - \omega_i^S \\
& \leq \sum_{i=1}^{m-1} -\bar\alpha_{i,+}^{S} - \,\bar\alpha_{i,-}^{S} -\bar\omega_i^S = \sum_{i=1}^{m-1}K\vol(\bar T_i), 
\end{align*}
where the second inequality is due to angle comparison, cf.\ Definition \ref{def: angle and 4pt}(i), and the last equality is due to Lemma \ref{lem: angle sum minkowski}. 
Clearly, $\vol(\bar T_i)$ is bounded from above by the volume of a timelike diamond $I(p,q)$ in $\lm{K}$ with $\ell(p,q)=\delta$. 
In particular, it goes to zero as $\delta \to 0$. 
Finally, applying $\limsup_{\delta \to 0}$ on both sides of $\kappa^*(\sigma)  \leq \kappa^*(\gamma) + \sum_{i=1}^{m-1} K\vol(\bar T_i)$, we obtain
\begin{equation*}
\TC(\gamma) \leq \kappa^*(\gamma). \qedhere
\end{equation*}
\end{proof}

\section{Length estimates for timelike curves} 
The goal for this section is to prove our main theorem and to show that timelike curves of finite total curvature are rectifiable, a result of independent interest.
To do so, we start out by providing explicit formulae for the minimal curvature of bi-segments in the respective model spaces and show that this minimum is attained precisely by (geodesics and) isosceles bi-segments.

\begin{lemma}[Length of bi-segments and rigidity]
\label{lem: bi-segment and rigidity}
Let $p\ll q\ll w$ be three timelike related points in $\lm{K}$. 
Let $\kappa=\ma_{q}(p,w)$ be the total curvature of the corresponding bi-segment, and let $r=\ell(p,w)<D_K$. 
Assume that $\ell(p,q)+\ell(q,w)=l_0$ is fixed. 
Then
\begin{equation*}
\label{total bound}
\kappa\geq \bm{\kappa}(K,r,l_0):=
\begin{cases}
2\arcosh\left(\dfrac{\sin(\sqrt{K}r/2)}{\sin(\sqrt{K}l_0/2)}\right), 
& \mathrm{if}\,\,K>0,\\[1.2ex]
2\arcosh\left(\dfrac{r}{l_0}\right), 
& \mathrm{if}\,\,K=0,\\[1.2ex]
2\arcosh\left(\dfrac{\sinh(\sqrt{-K}r/2)}{\sinh(\sqrt{-K}l_0/2)}\right), 
& \mathrm{if}\,\,K<0.
\end{cases}
\end{equation*}
Equality holds if and only if $\ell(p,q)=\ell(q,w)=l_0/2$ when $\kappa>0$, and if and only if $p,q,w$ lie on a geodesic in the case $\kappa=0$. Moreover, if $\kappa_0\in[0,+\infty)$ denotes the total curvature of the corresponding isosceles bi-segment, then its length is given by
\begin{equation}
\label{length bound}
\mathcal{L}(K,r,\kappa_0):=
\begin{cases}
\dfrac{2}{\sqrt{K}}\arcsin\left(\dfrac{\sin(\sqrt{K}r/2)}{\cosh(\kappa_0/2)}\right), & \mathrm{if}\,\,K>0,\\[1.2ex]
\dfrac{r}{\cosh(\kappa_0/2)}, & \mathrm{if}\,\,K=0,\\[1.2ex]
\dfrac{2}{\sqrt{-K}}\operatorname{arsinh}\left(\dfrac{\sinh(\sqrt{-K}r/2)}{\cosh(\kappa_0/2)}\right), & \mathrm{if}\,\,K<0.
\end{cases}
\end{equation}
\end{lemma}

\begin{proof}
We treat only the cases $K=0$ and $K<0$, since the case $K>0$ is analogous to the case $K<0$. We set
\[
u:=\ell(p,q)=m+x,\qquad v:=\ell(q,w)=m-x,\qquad m:=\frac{l_0}{2},\qquad -m<x<m.
\]

\medskip
\noindent
\textbf{Case 1: $K=0$.}
The Law of Cosines gives
\[
\cosh \kappa
=
\frac{r^2-2m^2-2x^2}{2(m^2-x^2)},
\]
which implies
\[
\frac{\mathrm{d}}{\mathrm{d}x}\cosh \kappa
=\frac{x(r^2-4m^2)}{(m^2-x^2)^2}.
\]
This vanishes either if $r=l_0=2m$, in which case $p,q$ and $w$ lie on a geodesic and $\kappa=0$, or if $r > l_0$ and $x=0$, in which case 
\[
\bm{\kappa}(K,r,l_0)=\arcosh\!\left(\frac{2r^2-l_0^2}{l_0^2}\right)=2\arcosh\!\left(\frac{r}{l_0}\right).
\]

\medskip
\noindent
\textbf{Case 2: $K<0$.}
Set $\lambda:=\sqrt{-K}$.
In the non-trivial case, in which the three points do not lie on a geodesic, we have $r>l_0=2m$.
By the Law of Cosines,
\begin{equation*}
\cosh\kappa = \frac{\cosh(\lambda r)-\cosh(\lambda u)\cosh(\lambda v)}{\sinh(\lambda u)\sinh(\lambda v)}.
\end{equation*}
Using $u+v=2m$ and $u-v=2x$, this becomes
\begin{equation*}
\cosh\kappa = \frac{\cosh(\lambda r)-\cosh^2(\lambda m)-\sinh^2(\lambda x)}{\sinh^2(\lambda m)-\sinh^2(\lambda x)}.
\end{equation*}
Now set
\begin{equation*}
y(x):=\sinh^2(\lambda x).
\end{equation*}
Then
\begin{equation*}
\cosh\kappa = \frac{\cosh(\lambda r)-\cosh^2(\lambda m)-y}{\sinh^2(\lambda m)-y}.
\end{equation*}
Differentiating with respect to $y$ gives
\begin{equation*}
\frac{\mathrm{d}}{\mathrm{d}y}\cosh\kappa = \frac{\cosh(\lambda r)-\cosh(2\lambda m)}{(\sinh^2(\lambda m)-y)^2} >0,
\end{equation*}
because $r>2m$. Hence $\cosh\kappa$, and therefore $\kappa$, is minimised exactly when $y(x)$ is minimised. Since
\begin{equation*}
y(x)=\sinh^2(\lambda x)
\end{equation*}
has a unique minimum at $x=0$, the minimum occurs exactly when
\begin{equation*}
u=v=m=\frac{l_0}{2}.
\end{equation*}
Substituting $x=0$ gives
\begin{equation*}
\begin{aligned}
\bm{\kappa}(K,r,l_0) = \arcosh{\left(\frac{\cosh(\sqrt{-K}r)-\cosh^2(\sqrt{-K}l_0/2)}{\sinh^2(\sqrt{-K}l_0/2)}\right)}= 2\arcosh{\left(\frac{\sinh(\sqrt{-K}r/2)}{\sinh(\sqrt{-K}l_0/2)}\right)}.
\end{aligned}
\end{equation*}
\noindent
The second assertion follows by a direct computation via the Law of Cosines.
\end{proof}

\begin{lemma}[Eliminating break-points reduces total curvature]
\label{deformation2}
Let $K \in \R$ and let $p_1 \ll p_2 \ll p_3 \ll p_4$ form a timelike poly-segment $\sigma$ in $\lm{K}$ such that $\ell(p_1,p_2) + \ell(p_2,p_3) < D_K/2$. 
Assume that $p_1$ and $p_4$ lie in the same closed half-space determined by the geodesic extending $[p_2,p_3]$.
Then $\sigma$ can be deformed at the vertex $p_2$ (fixing all other vertices) without changing its total length in such a way that $\ell(p_1,p_2)$ increases until the bi-segment formed by $[p_2,p_3]$ and $[p_3,p_4]$ becomes a geodesic, and such that the total curvature does not increase in the process.
Moreover, if $\ma_{p_2}(p_1,p_3)>0$, this deformation strictly decreases the total curvature.
\end{lemma}
\begin{figure}[htbp]
\begin{center}
\begin{tikzpicture}

\draw [] (0,0)-- (0,4);
\draw [] (0,4)-- (1,5.5);
\draw [dashed] (0,4)-- (-2,1);
\draw [] (0,0)-- (-0.6954195030509164,2.056668223940556);
\draw [] (-0.6954195030509164,2.056668223940556)-- (0,4);

\begin{scriptsize}
\coordinate [circle, fill=black, inner sep=0.5pt, label=270: {$p_1$}] (p1) at (0,0);
\coordinate [circle, fill=black, inner sep=0.5pt, label=225: {$p_2$}] (p2) at (-0.6954195030509164,2.056668223940556);
\coordinate [circle, fill=black, inner sep=0.5pt, label=315: {$p_3$}] (p3) at (0,4);
\coordinate [circle, fill=black, inner sep=0.5pt, label=225: {$p_2'$}] (p2') at (-0.6574152550914241,2.613724944413358);
\coordinate [circle, fill=black, inner sep=0.5pt, label=135: {$p_2''$}] (p2'') at (-0.5356970351018242,3.196454447347264);
\coordinate [circle, fill=black, inner sep=0.5pt, label=135: {$p_4$}] (p4) at (1,5.5);
\end{scriptsize}

\draw[dotted] (p1) -- (p2') -- (p3);
\draw[dotted] (p1) -- (p2'') -- (p3);

\end{tikzpicture}
\end{center}
\caption{The point $p_2$ moving along the deformation, reaching its final position at $p_2''$ on the extension of the geodesic segment $[p_3,p_4]$.}
\label{fig: deformation}
\end{figure}
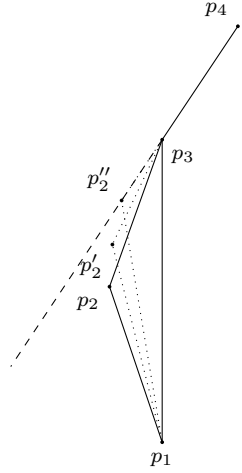

\begin{proof}
Set
\begin{equation*}
\begin{aligned}
&\alpha:=\ma_{p_3}(p_1,p_4),
&&\beta:=\ma_{p_2}(p_1,p_3),
&&\gamma:=\ma_{p_3}(p_1,p_2),
&&\eta\;:=\ma_{p_3}(p_2,p_4),\\
&a\,:=\ell(p_2,p_3), &&b\,:=\ell(p_1,p_3), &&c\,:=\ell(p_1,p_2),&&m:=\tfrac{1}{2}(a+c).
\end{aligned}
\end{equation*}
If $\beta = 0$, it follows that $p_2 \in [p_1,p_3]$. 
This forces the deformation to simply move $p_2$ up along the segment $[p_1,p_3]$ until it coincides with $p_3$. 
In particular, the total curvature remains constant throughout this deformation. 
Hence, it suffices to consider the non-trivial case $\beta > 0$, for which $b > a+c = 2m$.

We know that the deformation changes the total curvature by varying the sum $\beta+\eta$. 
Since the points $p_1$ and $p_4$ lie in the same closed half-space determined by the geodesic extending $[p_2,p_3]$, and since the elementary relation $\eta=\alpha-\gamma$ holds, with $\alpha$ fixed, it suffices to study the monotonicity of $\beta-\gamma$. 

\medskip
\noindent
\textbf{Case 1: $K=0$.}
We introduce a parameter $x$ such that
\begin{equation*}
c=m+x,\qquad a=m-x,\qquad -m< x\leq x_{\mathrm{max}}< m,
\end{equation*}
where $x_{\max}$ is attained when $[p_2,p_3]$ and $[p_3,p_4]$ become geodesics. Our claim is that
\begin{equation*}
x \mapsto \beta(x) - \gamma(x)
\end{equation*}
is strictly decreasing in the interval $(-m, x_{\textnormal{max}}]$ provided $\beta >0$, in which case $b>2m>2x$.
By the Law of Cosines, we have
\begin{equation*}
\cosh\beta(x) = \frac{b^2-2m^2-2x^2}{2(m^2-x^2)},\qquad \cosh\gamma(x) = \frac{b^2-4mx}{2b(m-x)}.
\end{equation*}
Elementary computations then give 
\begin{equation*}
\begin{aligned}
\beta'(x)&=\frac{\dfrac{\mathrm{d}}{\mathrm{d}x}\cosh\beta(x)}{\sinh\beta(x)}=
\frac{2x}{(m+x)(m-x)}\sqrt{\frac{b^2-4m^2}{b^2-4x^2}},\\
\gamma'(x)&=\frac{\dfrac{\mathrm{d}}{\mathrm{d}x}\cosh\gamma(x)}{\sinh\gamma(x)}=\frac{1}{m-x}\sqrt{\frac{b^2-4m^2}{b^2-4x^2}}.
\end{aligned}
\end{equation*}
Therefore,
\[
\beta'(x)-\gamma'(x)
=-\sqrt{\frac{b^2-4m^2}{b^2-4x^2}}
\frac{1}{m+x}< 0.
\]
Hence, $\beta(x) - \gamma(x)$ is indeed strictly decreasing in $x\in (-m,x_{\mathrm{max}}]$.

\medskip
\noindent
\textbf{Case 2: $K<0$.} Set $\lambda:=\sqrt{-K}$. As above, write
\begin{equation*}
c=m+x,\qquad a=m-x,\qquad -m<x\le x_{\max}<m.
\end{equation*}
In the non-trivial case, we have $b>2m$. We again compute $(\beta-\gamma)'(x)$ and argue that it is strictly negative. By the Law of Cosines,
\begin{equation*}
\cosh\beta = \frac{\cosh(\lambda b)-\cosh(\lambda a)\cosh(\lambda c)}
{\sinh(\lambda a)\sinh(\lambda c)}, 
\qquad
\cosh\gamma = \frac{\cosh(\lambda a)\cosh(\lambda b)-\cosh(\lambda c)}
{\sinh(\lambda a)\sinh(\lambda b)}.    
\end{equation*}
We set
\begin{equation*}
q(x):=\sqrt{\frac{\sinh\left(\frac{\lambda b+2\lambda m}{2}\right)\sinh\left(\frac{\lambda b-2\lambda m}{2}\right)}{\sinh\left(\frac{\lambda b+2\lambda x}{2}\right)\sinh\left(\frac{\lambda b-2\lambda x}{2}\right)}} > 0.
\end{equation*}
A computation analogous to the one above then gives
\begin{equation*}
\beta'(x)=\lambda q(x)
\frac{\sinh(2\lambda x)}
{\sinh(\lambda c)}\frac{1}{\sinh(\lambda a)},
\qquad
\gamma'(x) = \lambda q(x) \frac{1}{\sinh(\lambda a)}.
\end{equation*}
Hence,
\begin{equation*}
(\beta-\gamma)'(x)=\frac{\lambda q(x)}{\sinh(\lambda a)}\left(\frac{\sinh(2\lambda x)}{\sinh(\lambda c)}-1\right).
\end{equation*}
Since $c=m+x>2x$, we have
\begin{equation*}
\sinh(\lambda c)>\sinh(2\lambda x),
\end{equation*}
and thus
\begin{equation*}
(\beta-\gamma)'(x)<0,
\end{equation*}
implying that $\beta-\gamma$ is strictly decreasing throughout the deformation.

\medskip
\noindent
\textbf{Case 3: $K>0$.} Using notation similar to that in \textbf{Case 2}. We set
\begin{equation*}
\tilde{q}(x):=
\sqrt{\frac{\sin\left(\frac{\tilde{\lambda} b+2\tilde{\lambda} m}{2}\right)
\sin\left(\frac{\tilde{\lambda} b-2\tilde{\lambda} m}{2}\right)}{\sin\left(\frac{\tilde{\lambda} b+2\tilde{\lambda} x}{2}\right)
\sin\left(\frac{\tilde{\lambda} b-2\tilde{\lambda} x}{2}\right)}}>0,\quad \mathrm{where}\,\, \tilde{\lambda}:=\sqrt{K}.
\end{equation*}
Since $\ell(p_1,p_2) + \ell(p_2,p_3) < \frac{D_K}{2}$, a similar computation ultimately yields
\begin{equation}
\label{eq:K-positive-beta-gamma}
(\beta-\gamma)'(x)=
\frac{\tilde{\lambda} \tilde{q}(x)}{\sin(\tilde{\lambda} a)}
\left(\frac{\sin(2\tilde{\lambda} x)}{\sin(\tilde{\lambda} c)}-1 \right),
\end{equation}
from which the claim readily follows. 
\end{proof}

A time-reversed analogue, where $p_3$ moves while increasing $\ell(p_3,p_4)$ until $[p_1,p_2]$ and $[p_2,p_3]$ become a geodesic, clearly holds by the same arguments.
In the proof of Theorem \ref{main theorem}, we will iteratively apply the previous result to poly-segments in the respective model spaces.
This, however, will cause technical difficulties in the positive curvature case as \eqref{eq:K-positive-beta-gamma} fails to hold for distances larger than $\frac{D_K}{2}$.
To avoid overloading the proof of our main theorem, we prepare the following lemma:

\begin{lemma}[Direct estimate for $K > 0$]
\label{lem:positive_tri_segment_estimate}
Let $K>0$ and let $p_0\ll p_1\ll p_2\ll p_3$ be a convex timelike poly-segment consisting of three segments in $\lm{K}$ with total length $l$ and $ r:=\ell(p_0,p_3) < D_K$.\footnote{Here, convex means that the quadrilateral generated by these four vertices is convex.} Assume $\theta=\ma_{p_1}(p_0,p_2)>0,
\varphi=\ma_{p_2}(p_1,p_3)>0,$ and the total curvature $\Theta=\theta+\varphi$.
Then there exists an isosceles timelike bi-segment $B$ in $\lm{K}$ with the same endpoints
$p_0$ and $p_3$, total length $l$, and
\[
\kappa^*(B)< \Theta .
\]
\end{lemma}

\begin{proof}
Set
\begin{equation*}
\begin{aligned}
&a:=\ell(p_0,p_1), &&b:=\ell(p_1,p_2), &&c\;:=\ell(p_2,p_3), 
&d:=\ell(p_0,p_2), &&\psi:=\ma_{p_2}(p_0,p_1).
\end{aligned}
\end{equation*}
We claim
\begin{equation*}
\sin\left(\frac{\sqrt K\,r}{2}\right)
< \cosh\left(\frac{\Theta}{2}\right)
\sin\left(\frac{\sqrt K\,l}{2}\right).
\end{equation*}

By rescaling, it suffices to prove the claim for $K=1$, in which case $D_K=\pi$.  
Moreover, all side lengths appearing below are strictly smaller than $\pi$, and their $\sin$-values are positive.
By the Law of Cosines applied to the triangle $\Delta(p_0,p_1,p_2)$ at the vertices $p_1$ and $p_2$, respectively, we obtain 
\begin{equation}\label{eq add 1}
\cos d =
\cos a\,\cos b-\cosh\theta\,\sin a\, \sin b,\qquad \cos a =
\cos b\,\cos d+\cosh\psi\,\sin b\, \sin d.
\end{equation}
Combining these two identities yields
\begin{equation}\label{eq add 2}
\cosh\psi\,\sin d
=
\cos a\,\sin b+\cosh\theta\,\sin a\,\cos b.
\end{equation}
Moreover, solving for the respective $\cosh$-terms, squaring them and using $\cosh^2-\sinh^2=1$, we get
\begin{equation}\label{eq add 3}
\sinh\psi\,\sin d
=
\sinh\theta\,\sin a.
\end{equation}
Applying the Law of Cosines in $\Delta(p_0,p_2,p_3)$ at $p_2$ and using $\ma_{p_2}(p_0,p_3) = \ma_{p_2}(p_0,p_1) + \ma_{p_2}(p_1,p_3) = \psi + \varphi$, which holds by the assumed convexity, we obtain
\begin{equation*}
\cos r
=
\cos d\cos c-\cosh(\psi+\varphi)\sin d\sin c.    
\end{equation*}
Combining this with \eqref{eq add 1}, \eqref{eq add 2} and \eqref{eq add 3}, we get
\begin{equation*}
\begin{aligned}
\cos r
=\;&
\bigl(\cos a\cos b-\cosh\theta\,\sin a\sin b\bigr)\cos c\\
&-\cosh\varphi\,
\bigl(\cos a\sin b+\cosh\theta\,\sin a\cos b\bigr)\sin c\\
&-\sinh\theta\sinh\varphi\,\sin a\sin c.
\end{aligned}
\end{equation*}
Then, using the identity
\begin{equation*}
\cos(a+b+c)
=
\cos a\,\cos b\,\cos c
-\sin a\,\sin b\,\cos c
-\cos a\,\sin b\,\sin c
-\sin a\,\cos b\,\sin c,
\end{equation*}
we obtain
\begin{equation*}
\begin{aligned}
\cos r
=\;&
\cos(a+b+c)
-(\cosh\theta-1)\sin a\,\sin(b+c)\\
&-(\cosh\varphi-1)\sin c\,\sin(a+b)\\
&-\Big((\cosh\theta-1)(\cosh\varphi-1)\cos b
+\sinh\theta\sinh\varphi\Big)\sin a\,\sin c.
\end{aligned}
\end{equation*}
We now estimate the right-hand side. Since $a+b+c=l<\pi,$
we have
\begin{equation*}
\sin a\, \sin(b+c)\leq \sin^2\frac{l}{2},
\qquad
\sin c\, \sin(a+b)\leq \sin^2\frac{l}{2},\qquad \sin a\, \sin c< \sin^2\frac{l}{2}.
\end{equation*}
Set
\begin{equation*}
C:=(\cosh\theta-1)(\cosh\varphi-1)\cos b
+\sinh\theta\sinh\varphi.
\end{equation*}
Then $C > 0$. Indeed,
\begin{equation*}
\begin{aligned}
C&>
-(\cosh\theta-1)(\cosh\varphi-1)
+\sinh\theta\sinh\varphi\\
&=
4\sinh\frac\theta2\sinh\frac{\varphi}{2}
\cosh\frac{\theta-\varphi}{2}
>0.
\end{aligned}
\end{equation*}
Moreover, 
\begin{equation*}
C<(\cosh\theta-1)(\cosh\varphi-1)
+\sinh\theta\sinh\varphi.
\end{equation*}
Therefore,
\begin{equation*}
\begin{aligned}
\cos r
&>
\cos l-\sin^2\frac{l}{2}
\Big[(\cosh\theta-1)+(\cosh\varphi-1)
+(\cosh\theta-1)(\cosh\varphi-1)
+\sinh\theta\sinh\varphi\Big]\\
&=\cos l-(\cosh\Theta-1)\sin^2\frac{l}{2}\\
&=1-2\cosh^2\frac{\Theta}{2}\,\sin^2\frac{l}{2}.
\end{aligned}
\end{equation*}
Equivalently,
\begin{equation*}
1-2\sin^2\frac{r}{2}
>
1-2\cosh^2\frac{\Theta}{2}\,\sin^2\frac{l}{2},
\end{equation*}
and hence
\begin{equation*}
\sin\frac{r}{2}<
\cosh\frac{\Theta}{2}\,\sin\frac{l}{2}.
\end{equation*}
Reintroducing the factor $\sqrt K$, we get
\begin{equation*}
\sin\left(\frac{\sqrt K\,r}{2}\right)<\cosh\left(\frac{\Theta}{2}\right)
\sin\left(\frac{\sqrt K\,l}{2}\right),
\end{equation*}
by the reverse triangle inequality, we know that $l\leq r<D_K$, so the function $t\mapsto \sin(\sqrt K\,t/2)$ is increasing on
$[0,D_K]$, and therefore
\begin{equation*}
1\leq\frac{\sin(\sqrt K\,r/2)}{\sin(\sqrt K\,l/2)}.
\end{equation*}
Defining
\begin{equation*}
\kappa_0:=2\,\arcosh\left(
\frac{\sin(\sqrt K\,r/2)}{\sin(\sqrt K\,l/2)}\right),
\end{equation*}
we get $\kappa_0< \Theta$. 
Now, by Lemma \ref{lem: bi-segment and rigidity}, $\kappa_0$ is precisely the total curvature of an isosceles bi-segment where the endpoints are a distance $r$ apart, and the sides are each of length $l/2$. \qedhere


\end{proof}




Next, we show that a timelike curve with finite total curvature must be rectifiable. 
To this end, we first prove a characterisation of the $\ell$-length of a timelike curve as a limit of total $\ell$-variations with respect to mesh-size, akin to the metric result.

\begin{lemma}[Length as limit]
\label{Length as limit}
Let $X$ be a \LpLS and let $\gamma \colon[a,b] \to X$ be a timelike curve. 
Suppose that $\ell$ is continuous. 
Then 
\begin{equation*}
L(\gamma)=\lim_{|\sigma|\to 0}\sum_{i=0}^{N-1}\ell\big(\gamma(t_i),\gamma(t_{i+1})\big),
\end{equation*}
where $\sigma=\{a=t_0 < t_1 < \ldots < t_N=b\}$ is a partition of $[a,b]$ and $|\sigma|:=\max_{0\le i\le N-1}(t_{i+1}-t_i)$.
\end{lemma}

\begin{proof}
For a partition $\sigma=\{a=t_0<\cdots<t_N=b\}$, denote
\[
V_\sigma(\gamma):=\sum_{i=0}^{N-1}\ell\big(\gamma(t_i),\gamma(t_{i+1})\big).
\]
By the definition of $L$, see \eqref{definition of length}, we know that  $L(\gamma)\leq V_\sigma(\gamma)$ for every partition $\sigma$. Hence
\begin{equation}\label{liminf of length}
L(\gamma)\leq \liminf_{|\sigma|\to 0}V_\sigma(\gamma).
\end{equation}
It therefore remains to show that for every $M>L(\gamma)$, there exists $\delta>0$ such that for all partitions $\sigma'$ with $|\sigma'|<\delta$, one has $V_{\sigma'}(\gamma)<M.$

Let $M>L(\gamma)$. Choose $\varepsilon>0$ such that $L(\gamma)<M-\varepsilon$. 
By the definition of $L(\gamma)$, there exists a partition 
$\sigma=\{a=t_0<t_1<\ldots<t_N=b\}$ such that $V_\sigma(\gamma)<M-\varepsilon$. 
Refining $\sigma$, if necessary, we may assume that $N\ge 2$; by the reverse triangle inequality, this does not increase $V_\sigma(\gamma)$.
Since $\gamma([a,b])$ is compact and $\ell$ is continuous, $\ell$ is uniformly continuous on 
$\gamma([a,b])\times\gamma([a,b])$. Hence there exists $\delta'>0$ such that 
\[
\ell\big(\gamma(s),\gamma(t)\big)<\frac{\eps}{2(N-1)},\qquad \mathrm{when}\,\,\,\mathrm{d}\big(\gamma(s),\gamma(t)\big)<\delta'.
\]
Further, as $\gamma$ is continuous on the compact interval $[a,b]$, it is uniformly continuous, and hence there exists $\delta_0>0$ such that 
\[
\mathrm{d}\big(\gamma(s),\gamma(t)\big)<\delta',\qquad \mathrm{when}\,\,\,|t-s|<\delta_0.
\]
Now set
\[
\delta:=\min\left\{\frac{\delta_0}{2},\frac{m}{4}\right\},\qquad \mathrm{where}\,\,m:=\min_{0\le i\le N-1}(t_{i+1}-t_i)>0.
\]
Let now $\sigma'=\{a=s_0<\cdots<s_{N'}=b\}$ be a partition of $[a,b]$ with $|\sigma'|<\delta$. 
For each $i=1,\ldots,N-1$, denote by $t_i'$ the vertex of $\sigma'$ that is closest to $t_i$ among those that satisfy $t_i'\le t_i$, and by $t_i''$ the vertex of $\sigma'$ that is closest to $t_i$ among those that satisfy $t_i''>t_i$. Then $t_i'\le t_i<t_i''$, and since $|\sigma'|<m/4$, we have
\[
t_i'\le t_i<t_i''<t_{i+1}'\le t_{i+1}<t_{i+1}'',\qquad \forall \,i=1,\ldots,N-2.
\]
Moreover,
\[
t_i''-t_i'\le (t_i''-t_i)+(t_i-t_i')<2|\sigma'|<2\delta\le \delta_0,
\]
and hence $\mathrm{d}(\gamma(t_i'),\gamma(t_i''))<\delta'.$
Therefore,
\[
\ell\big(\gamma(t_i'),\gamma(t_i'')\big)<\frac{\eps}{2(N-1)},
\qquad\forall \,i=1,\ldots,N-1.
\]
Now compare the partitions $\sigma'$ and $\sigma\cup\sigma'$. The latter is obtained from $\sigma'$ by inserting the points $t_1,\ldots,t_{N-1}$. Hence
\begin{align*}
V_{\sigma'}(\gamma)-V_{\sigma\cup\sigma'}(\gamma)
&=
\sum_{i=1}^{N-1}
\Bigl(
\ell(\gamma(t_i'),\gamma(t_i''))
-\ell(\gamma(t_i'),\gamma(t_i))
-\ell(\gamma(t_i),\gamma(t_i''))
\Bigr) \\
&\le
\sum_{i=1}^{N-1}\ell\big(\gamma(t_i'),\gamma(t_i'')\big) \\
&<
(N-1)\frac{\eps}{2(N-1)}
=
\frac{\eps}{2}.
\end{align*}

\noindent
On the other hand, $\sigma\cup\sigma'$ is a refinement of $\sigma$, so by the reverse triangle inequality,
\[
V_{\sigma\cup\sigma'}(\gamma)\le V_\sigma(\gamma)<M-\eps.
\]
Therefore,
\[
V_{\sigma'}(\gamma)
<
V_{\sigma\cup\sigma'}(\gamma)+\frac{\eps}{2}
<
M-\eps+\frac{\eps}{2}
<
M.
\]
This proves that there exists $\delta>0$ such that $V_{\sigma'}(\gamma)<M$ whenever $|\sigma'|<\delta$. Hence,
\[
\limsup_{|\sigma|\to 0}V_\sigma(\gamma)\le L(\gamma).
\]
Combining this with \eqref{liminf of length}, the proof is complete.
\end{proof}

\begin{proposition}[Rectifiability]
\label{prop: rectifiable}
Let $X$ be a  strongly causal and regular \LpLS and let $K \in \R$. 
Let $\gamma \colon[a,b] \to X$ be a timelike curve inside a ($\leq K$)-comparison neighbourhood $U$ such that $\ell(\gamma(a),\gamma(b)) < D_K$.
If $\mathrm{TC}(\gamma)<+\infty,$ then $\gamma$ is rectifiable.
\end{proposition}

\begin{proof}
Assume, for the sake of contradiction, that $\gamma$ is not rectifiable. 
Then we can find $a \le s < t \le b$ such that $\eta:=\gamma|_{[s,t]}$ satisfies $L(\eta)=0$.
Since $\gamma$ is timelike, so is $\eta$. 
In particular, its endpoints $x:=\eta(s)$ and $y:=\eta(t)$ satisfy $r_0:=\ell(x,y)>0$. 

By Lemma \ref{Length as limit}, there exists a sequence of partitions whose mesh-sizes tend to zero and such that the corresponding $\ell$-sums converge to $L(\eta)=0$. 
Since $\eta$ is timelike and contained in a comparison neighbourhood, after possibly refining each partition, we may assume that consecutive partition points are joined by nonconstant $\ell$-maximising timelike geodesic segments. Thus, we obtain a sequence of timelike poly-segments $P_n$ inscribed in $\eta$ such that
\[
|P_n|\to 0,
\qquad
l_n:=L(P_n)\to 0.
\]
It is then not hard to see that, for each $n$, the timelike poly-segment $P_n$ is rectifiable. 
Moreover, since $l_n \to 0$, we have $l_n<D_K/2$ for large enough $n$.
Invoking the Majorisation Theorem \ref{thm: majorisation}, we therefore obtain a convex (and timelike) poly-segment $\widetilde P_n$ in $\lm{K}$ of the same length $L(\widetilde P_n)=L(P_n)=l_n$ and with the same time separation between endpoints as well as an anti-Lipschitz map from the region enclosed by $\widetilde P_n$ and $[\tx,\ty]$ into $U$ such that $\widetilde P_n$ and $[\tx,\ty]$ are, respectively, mapped onto $P_n$ and $[x,y]$ such that the $\ell$-length is preserved. 
Let $p,q,r$ denote any three subsequent vertices of $P_n$, and consider the triangle $\Delta(p,q,r)$, its comparison triangle $\Delta(\bp,\bq,\br)$ as well as the triangle $\Delta(\tp,\tq,\tr)$ which arises via $\widetilde P_n$. 
The side lengths involving the middle vertex are the same in all three triangles. 
By the anti-Lipschitz property of the majorising map, we obtain $\ell(\tp,\tr) \leq \ell(p,r)=\ell(\bp,\br)$, hence, by the Law of Cosines, $\ma_{\tq}(\tp,\tr) \leq \ma_{\bq}(\bp,\br)$. 
Moreover, by angle comparison, we have $\ma_{\bq}(\bp,\br) \leq \ma_q(p,r)$. 
Since this holds true for any three subsequent vertices, we thus have $\kappa^*(P_n) \geq \kappa^*(\widetilde P_n)$. 

We continue to work in the model space $\mathbb{L}^2(K)$. Since $l_n<D_K/2$, the length assumption in Lemma \ref{deformation2} is satisfied for every pair of adjacent segments throughout the reduction process. Applying Lemma \ref{deformation2} repeatedly, we obtain a timelike bi-segment $B_n$ of equal length, with the same time separation between its endpoints, and with no greater total curvature: 
\[
L(B_n)=L(\widetilde P_n)=l_n,\qquad
\kappa^*(B_n)\le \kappa^*(\widetilde P_n).
\]
Hence
\begin{equation*}
    \kappa^*(P_n)\geq \kappa^*(\widetilde P_n)\geq \kappa^*(B_n) \geq \bm{\kappa}(K,r_0,l_n),
\end{equation*}
where the last inequality follows from Lemma~\ref{lem: bi-segment and rigidity}.  Since $r_0>0$ is fixed and $l_n\to 0$, we have $\bm{\kappa}(K,r_0,l_n)\to+\infty$. Consequently, $\kappa^*(P_n)\to+\infty$.
Because $P_n$ is inscribed in $\eta$ and $|P_n|\to 0$, the definition of total curvature gives
\begin{equation*}
    \mathrm{TC}(\eta)\geq \lim_{n\to+\infty} \kappa^*(P_n)=+\infty.
\end{equation*}
Hence $\mathrm{TC}(\eta)=+\infty.$ Since any poly-segment inscribed in $\eta$ can be extended to one inscribed in $\gamma$ by adding partitions of $[a,s]$ and $[t,b]$ with mesh tending to zero, and since the additional angles are non-negative, we have 
\begin{equation*}
    \mathrm{TC}(\gamma)\geq \mathrm{TC}(\eta)=+\infty,
\end{equation*}
which is a contradiction.
\end{proof}

Finally, let us state and prove our main result. 

\begin{theorem}[Lower bound on the length of timelike curves]
\label{main theorem}
Let $X$ be a strongly causal and regular \LpLS and let $K \in \R$. 
Let $\gamma$ be a timelike curve inside a ($\leq K$)-comparison neighbourhood. 
Denote by $x$ and $y$ the endpoints of $\gamma$ and set $r=\ell(x,y)<D_K$ and $\kappa=\TC(\gamma)\in [0,+\infty)$. 
Then 
\begin{equation}
\label{eq: length bound thm}
    L(\gamma) \geq \mathcal{L}(K,r,\kappa).
\end{equation}
If $\gamma$ is a timelike poly-segment and attains equality in \eqref{eq: length bound thm}, then $\gamma$ is a timelike geodesic in the case $\kappa=0$, and is the $\ell$-length-preserving image of an isosceles timelike bi-segment in $\lm K$ with total curvature $\kappa$ if $\kappa>0$.
\end{theorem}

\begin{proof}
Suppose first that $\gamma$ is a poly-segment, say it consists of $m$ segments. Denote the vertices
by $p_0,\ldots,p_m$. Then $\kappa=\kappa^*(\gamma)$ by Proposition \ref{prop: curvatures agree}. 
Moreover, $\gamma$ is rectifiable by Proposition \ref{prop: rectifiable}. 
Thus, using Theorem \ref{thm: majorisation} as in the proof of Proposition \ref{prop: rectifiable}, we know that there exists a convex
poly-segment $\widetilde\gamma$ with $m$ segments in $\lm K$ such that $L(\gamma)=L(\widetilde\gamma)$ and
$\kappa=\kappa^*(\gamma)\geq \kappa^*(\widetilde\gamma)$. Set
$\widetilde\gamma_m:=\widetilde\gamma$, and denote its vertices by
$\widetilde p_0,\ldots,\widetilde p_m$. As long as the current number of segments is greater than $3$, it is always the case that either
\begin{equation*}     
    \ell(\widetilde p_0,\widetilde p_1)
    +\ell(\widetilde p_1,\widetilde p_2)<\frac{D_K}{2} \qquad
    \text{ or } \qquad
    \ell(\widetilde p_{m-2},\widetilde p_{m-1})
    +\ell(\widetilde p_{m-1},\widetilde p_m)<\frac{D_K}{2}.
\end{equation*}
Hence, we can apply Lemma \ref{deformation2} to that poly-segment.
This decreases the number of vertices by one while not increasing the total curvature. Denoting
the resulting curve by $\widetilde\gamma_{m-1}$, and relabelling its vertices as
$\widetilde p_0,\ldots,\widetilde p_{m-1}$, this gives
$\kappa^*(\widetilde\gamma_m)\ge \kappa^*(\widetilde\gamma_{m-1})$.

Just as in Proposition \ref{prop: rectifiable}, we can keep applying Lemma \ref{deformation2} until  we have reduced
$\widetilde\gamma_m$ to the poly-segment $\widetilde\gamma_3$ consisting of three segments, whose
vertices we denote by $\widetilde p_0,\ldots,\widetilde p_3$, where we distinguish the following two cases:
if $K>0$ and both
\begin{equation*}
    \ell(\widetilde p_0,\widetilde p_1)
    +\ell(\widetilde p_1,\widetilde p_2)\geq \frac{D_K}{2}
    \qquad \text{ and } \qquad
    \ell(\widetilde p_1,\widetilde p_2)
    +\ell(\widetilde p_2,\widetilde p_3)\geq \frac{D_K}{2},    
\end{equation*}
then we apply Lemma~\ref{lem:positive_tri_segment_estimate} directly to obtain an isosceles bi-segment $\tilde \alpha$ of the same length, with the same time separation between its endpoints, and with no greater total curvature. 
Otherwise, we may apply Lemma~\ref{deformation2} one more time to end up with a bi-segment $\tilde\gamma_2$ of equal length, with the same time separation between its endpoints, and with no greater total curvature. 
By Lemma~\ref{lem: bi-segment and rigidity}, we know
that among all such bi-segments, the isosceles one, still denoted by $\widetilde\alpha$, has the smallest total
curvature. That is, in summary, we have

\begin{equation}
\label{eq: estimates summary}
    \kappa = \kappa^*(\gamma) \geq \kappa^*(\tilde\gamma)=\kappa^*(\tilde\gamma_m) \geq \kappa^*(\tilde\gamma_{m-1}) \geq \ldots \geq \kappa^*(\tilde\gamma_3) \big[ \geq \kappa^*(\tilde\gamma_2) 
    \big] \geq \kappa^*(\tilde \alpha). 
\end{equation}

Now, the length of $\tilde \alpha$, which is equal to that of $\gamma$, as the above procedure preserves lengths, is explicitly given by $L(\tilde \alpha)=\mathcal{L}(K,r,\kappa^*(\tilde \alpha))$, cf.\ \eqref{length bound}.
Since $\mathcal{L}$ is decreasing with respect to the total curvature and $\kappa^*(\tilde \alpha)\le \kappa$, it follows that 
\begin{equation}\label{eq: LengthIneqMainThm}
    L(\gamma)= L(\tilde \alpha)=\mathcal{L}(K,r,\kappa^*(\tilde \alpha))\ge \mathcal{L}(K,r,\kappa).
\end{equation}

Let now $\gamma$ be any timelike curve. 
Consider a sequence of inscribed poly-segments $\sigma_n$ such that 
\begin{equation*}
    |\sigma_n| \to 0 \qquad \textnormal{and} \qquad \kappa^*(\sigma_n) \to \kappa.
\end{equation*} 
By definition, we have $L(\sigma_n) \geq L(\gamma)$, and each $\sigma_n$ has endpoints $x$ and $y$.
In particular, we also have that $L(\sigma_n) \leq \ell(x,y) < D_K$.
The above arguments give $L(\sigma_n) \geq \mathcal{L}(K,r,\kappa^*(\sigma_n))$.
By Lemma \ref{Length as limit}, we have $L(\sigma_n) \to L(\gamma)$, and so 
\begin{equation*} 
\label{eq: LengthApproximationMainThm}
    L(\gamma) = \lim_{n \rightarrow \infty}{L(\sigma_n)} \geq \lim_{n \rightarrow \infty}\mathcal{L}(K,r,\kappa^*(\sigma_n)) = \mathcal{L}(K,r,\kappa),
\end{equation*}
as $\mathcal{L}$ is clearly continuous in the total curvature. 

Finally, suppose $\gamma$ is a poly-segment and achieves equality in the length bound \eqref{eq: length bound thm}. 
If $\kappa=0$, then the poly-segment $\gamma$ is a geodesic and the claim follows immediately, so suppose that $\kappa >0$.
Equality in \eqref{eq: length bound thm} implies that we have equality in \eqref{eq: LengthIneqMainThm} and so $\kappa^*(\tilde \alpha) = \kappa$, as $\mathcal{L}$ is strictly decreasing in the total curvature.
Therefore, equality holds throughout \eqref{eq: estimates summary}, implying that $\tilde \gamma$ must have been a bi-segment to begin with, as otherwise at least one of the steps in the deformation of Lemma \ref{deformation2} or Lemma \ref{lem:positive_tri_segment_estimate} would have strictly decreased the curvature (ignoring vertices in the interior of a segment).
Moreover, since $\kappa^*(\tilde \gamma)=\kappa^*(\tilde \alpha)$, equality holds in the application of Lemma \ref{lem: bi-segment and rigidity}. By the equality statement in Lemma \ref{lem: bi-segment and rigidity}, the bi-segment $\tilde \gamma$ must therefore be isosceles, and so, $\gamma$ is indeed the $\ell$-length-preserving image of an isosceles timelike bi-segment as the boundary map in Theorem \ref{thm: majorisation} preserves $\ell$-length.
\end{proof}

\begin{remark}[Asymptotic rigidity]
If $\gamma$ is an arbitrary timelike curve that achieves equality in the length estimate \eqref{eq: length bound thm}, one can follow a similar argument as above along a sequence of approximating poly-segments $\sigma_n$ to conclude that $\gamma = \lim_{n \rightarrow \infty} f_n(\tilde \sigma_n)$ where $\tilde \sigma_n$ is a sequence of bi-segments in $\lm K$ converging to an isosceles bi-segment $\tilde\alpha$ and the maps $f_n$ are $\ell$-length preserving.
The full rigidity statement, i.e., the conclusion that $\gamma$ is the ($\ell$-length preserving) image of this isosceles bi-segment $\tilde \alpha$, would require that the maps $f_n \colon \tilde \alpha_n \rightarrow \sigma_n$ converge to a map $f \colon \tilde \alpha \rightarrow \gamma$, which is not guaranteed by Theorem \ref{thm: majorisation}. 
\end{remark}

\begin{remark}[On assumptions]
Note that, in general, the theorem fails to hold if the curve leaves the $(\leq K)$-comparison neighbourhood.
Indeed, consider the Lorentzian cylinder $\R \times \mathbb S^1$ with the induced intrinsic structure from 3-dimensional Minkowski space. 
Its curvature is locally bounded from above by zero, but not globally (this is because the cylinder is not future one-connected, cf.\ \cite{EGCartanHadamard}).
Following any null geodesic $\gamma$ that winds around the cylinder once, we have connected two timelike related points $x$ and $y$ (its endpoints) with a curve of zero length.
Slightly deforming this curve, we can produce a family of timelike geodesics $\gamma_n$ (with total curvature $\kappa = 0$) such that $\ell(\gamma_n(0),\gamma_n(1)) \geq \ell(x,y)$ and $L(\gamma_n) \rightarrow L(\gamma) =0$ while $\mathcal{L}(K,\ell(x,y),\kappa) = \ell(x,y) >0$. 
\end{remark}

If the curvature bound on $X$ is global, then the only restrictions for the existence of a lower bound on the length are that the time separation of the endpoints be less than $D_K$ and that the total curvature be finite.

\section*{Acknowledgements}
\setcounter{equation}{0}
We would like to thank Tobias Beran for insightful comments and discussion in the early stages of this work. 
This research was supported in part by the Austrian Science Fund (FWF) [Grants DOI \href{https://doi.org/10.55776/PAT1996423}{10.55776/PAT1996423} and \href{https://doi.org/10.55776/EFP6}{10.55776/EFP6}]. 
FR and ZX acknowledge the support of the European Union - NextGenerationEU, in the framework of the PRIN Project `Contemporary perspectives on geometry and gravity' (code 2022JJ8KER – CUP G53D23001810006). The views and opinions expressed are solely those of the authors and do not necessarily reflect those of the European Union, nor can the European Union be held responsible for them. 
ZX is also supported by the Young Scientist Programs of the Ministry of Science \& Technology of China (2021YFA1000900, 2021YFA1002200), the Shandong Provincial Natural Science Foundation (ZR2025QB05), and a China Scholarship Council grant (No.\allowbreak 202406340143).
For open access purposes, the authors have applied a CC BY public copyright license to any author-accepted manuscript version arising from this submission.

\bibliography{references} 

@article{ABG2010,
  title={Total curvature and simple pursuit on domains of curvature bounded above},
  author={Alexander, S and Bishop, R and Ghrist, R},
  journal={Geometriae Dedicata},
  volume={149},
  number={1},
  pages={275--290},
  year={2010},
  publisher={Springer}
}

@article {ML2003,
    AUTHOR = {Maneesawarng, Chaiwat and Lenbury, Yongwimon},
     TITLE = {Total curvature and length estimate for curves in {${\rm
              CAT}(K)$} spaces},
   JOURNAL = {Differential Geom. Appl.},
  FJOURNAL = {Differential Geometry and its Applications},
    VOLUME = {19},
      YEAR = {2003},
    NUMBER = {2},
     PAGES = {211--222},
      ISSN = {0926-2245,1872-6984},
   MRCLASS = {53C22},
  MRNUMBER = {2002660},
MRREVIEWER = {Koichi\ Nagano},
       DOI = {10.1016/S0926-2245(03)00031-7},
       URL = {https://doi.org/10.1016/S0926-2245(03)00031-7},
}

@article {LMM2010,
    AUTHOR = {Castrill\'on L\'opez, M. and Fern\'andez Mateos, V. and
              Mu\~noz Masqu\'e, J.},
     TITLE = {Total curvature of curves in {R}iemannian manifolds},
   JOURNAL = {Differential Geom. Appl.},
  FJOURNAL = {Differential Geometry and its Applications},
    VOLUME = {28},
      YEAR = {2010},
    NUMBER = {2},
     PAGES = {140--147},
      ISSN = {0926-2245,1872-6984},
   MRCLASS = {53C22 (53C20)},
  MRNUMBER = {2594458},
MRREVIEWER = {Tobias\ Ekholm},
       DOI = {10.1016/j.difgeo.2009.10.008},
       URL = {https://doi.org/10.1016/j.difgeo.2009.10.008},
}

@misc{RXZ26+,
      title={Lorentz meets {P}tolemy}, 
      author={Felix Rott and Zhe-Feng Xu and Matteo Zanardini},
      year={2026},
      eprint={2601.21596},
      archivePrefix={arXiv},
      primaryClass={math.DG},
      url={https://arxiv.org/abs/2601.21596}, 
    note={Preprint: \url{https://arxiv.org/abs/2601.21596}}
}

@article {BKR24,
    AUTHOR = {Beran, Tobias and Kunzinger, Michael and Rott, Felix},
     TITLE = {On curvature bounds in {L}orentzian length spaces},
   JOURNAL = {J. Lond. Math. Soc. (2)},
  FJOURNAL = {Journal of the London Mathematical Society. Second Series},
    VOLUME = {110},
      YEAR = {2024},
    NUMBER = {2},
     PAGES = {Paper No. e12971, 41},
      ISSN = {0024-6107,1469-7750},
   MRCLASS = {53C23 (53B30 53C50)},
  MRNUMBER = {4781260},
MRREVIEWER = {Argam\ Ohanyan},
       DOI = {10.1112/jlms.12971},
       URL = {https://doi.org/10.1112/jlms.12971},
}

@article {BFG15,
    AUTHOR = {Barros, Manuel and Ferr\'andez, Angel and Garay, \'Oscar J.},
     TITLE = {Extremal curves of the total curvature in homogeneous
              3-spaces},
   JOURNAL = {J. Math. Anal. Appl.},
  FJOURNAL = {Journal of Mathematical Analysis and Applications},
    VOLUME = {431},
      YEAR = {2015},
    NUMBER = {1},
     PAGES = {342--364},
      ISSN = {0022-247X,1096-0813},
   MRCLASS = {53A04 (53C30)},
  MRNUMBER = {3357589},
MRREVIEWER = {Gary\ R.\ Jensen},
       DOI = {10.1016/j.jmaa.2015.05.072},
       URL = {https://doi.org/10.1016/j.jmaa.2015.05.072},
}

@article {FGL01,
    AUTHOR = {Ferr\'andez, Angel and Gim\'enez, Angel and Lucas, Pascual},
     TITLE = {Null helices in {L}orentzian space forms},
   JOURNAL = {Internat. J. Modern Phys. A},
  FJOURNAL = {International Journal of Modern Physics A. Particles and
              Fields. Gravitation. Cosmology. Astrophysics. Accelerator
              Physics},
    VOLUME = {16},
      YEAR = {2001},
    NUMBER = {30},
     PAGES = {4845--4863},
      ISSN = {0217-751X,1793-656X},
   MRCLASS = {53C50 (53C40)},
  MRNUMBER = {1873162},
       DOI = {10.1142/S0217751X01005821},
       URL = {https://doi.org/10.1142/S0217751X01005821},
}

@article {BS23,
    AUTHOR = {Beran, Tobias and S\"amann, Clemens},
     TITLE = {Hyperbolic angles in {L}orentzian length spaces and timelike
              curvature bounds},
   JOURNAL = {J. Lond. Math. Soc. (2)},
  FJOURNAL = {Journal of the London Mathematical Society. Second Series},
    VOLUME = {107},
      YEAR = {2023},
    NUMBER = {5},
     PAGES = {1823--1880},
      ISSN = {0024-6107,1469-7750},
   MRCLASS = {53B30 (28A75 51K10 53C23 53C50 53C80)},
  MRNUMBER = {4585303},
MRREVIEWER = {Benjam\'in\ Olea},
       DOI = {10.1112/jlms.12726},
       URL = {https://doi.org/10.1112/jlms.12726},
}

@misc{BR25+,
      title={Reshetnyak {M}ajorisation and discrete upper curvature bounds for {L}orentzian length spaces}, 
      author={Tobias Beran and Felix Rott},
      year={2025},
      eprint={2509.05224},
      archivePrefix={arXiv},
      primaryClass={math.DG},
      url={https://arxiv.org/abs/2509.05224}, 
    note={Preprint: \url{https://arxiv.org/abs/2509.05224}}
}

@article {EGCartanHadamard,
    AUTHOR = {Er\"os, Darius and Gieger, Sebastian},
     TITLE = {A synthetic {L}orentzian {C}artan-{H}adamard theorem},
   JOURNAL = {Trans. Amer. Math. Soc. Ser. B},
  FJOURNAL = {Transactions of the American Mathematical Society. Series B},
    VOLUME = {13},
      YEAR = {2026},
     PAGES = {132--155},
      ISSN = {2330-0000},
   MRCLASS = {53C50 (51K10 53B30 53C23)},
  MRNUMBER = {5062748},
       DOI = {10.1090/btran/248},
       URL = {https://doi.org/10.1090/btran/248},
}

@article {KS18,
    AUTHOR = {Kunzinger, Michael and S\"amann, Clemens},
     TITLE = {Lorentzian length spaces},
   JOURNAL = {Ann. Global Anal. Geom.},
  FJOURNAL = {Annals of Global Analysis and Geometry},
    VOLUME = {54},
      YEAR = {2018},
    NUMBER = {3},
     PAGES = {399--447},
      ISSN = {0232-704X,1572-9060},
   MRCLASS = {53C23 (53B30 53C50 53C80)},
  MRNUMBER = {3867652},
MRREVIEWER = {Benjam\'in\ Olea},
       DOI = {10.1007/s10455-018-9633-1},
       URL = {https://doi.org/10.1007/s10455-018-9633-1},
}

@article {AB08,
    AUTHOR = {Alexander, Stephanie B. and Bishop, Richard L.},
     TITLE = {Lorentz and semi-{R}iemannian spaces with {A}lexandrov
              curvature bounds},
   JOURNAL = {Comm. Anal. Geom.},
  FJOURNAL = {Communications in Analysis and Geometry},
    VOLUME = {16},
      YEAR = {2008},
    NUMBER = {2},
     PAGES = {251--282},
      ISSN = {1019-8385,1944-9992},
   MRCLASS = {53C50 (53B30 53C21)},
  MRNUMBER = {2425468},
MRREVIEWER = {J.\ Carlos\ D\'iaz-Ramos},
       DOI = {10.4310/cag.2008.v16.n2.a1},
       URL = {https://doi.org/10.4310/cag.2008.v16.n2.a1},
}

@article {BNR25,
    AUTHOR = {Beran, Tobias and Napper, Lewis and Rott, Felix},
     TITLE = {Alexandrov's patchwork and the {B}onnet-{M}yers theorem for
              {L}orentzian length spaces},
   JOURNAL = {Trans. Amer. Math. Soc.},
  FJOURNAL = {Transactions of the American Mathematical Society},
    VOLUME = {378},
      YEAR = {2025},
    NUMBER = {4},
     PAGES = {2713--2743},
      ISSN = {0002-9947,1088-6850},
   MRCLASS = {53C50 (51K10 53B30 53C23)},
  MRNUMBER = {4880460},
       DOI = {10.1090/tran/9372},
       URL = {https://doi.org/10.1090/tran/9372},
}

@misc{BBCGRR26+,
      title={A {S}plitting {T}heorem for non-positively curved {L}orentzian spaces}, 
      author={Barton, Joe and Beran, Tobias and Che, Mauricio and Gieger, Sebastian and R\"ohrig, Jona and Rott, Felix},
      year={2026},
      eprint={2601.14058},
      archivePrefix={arXiv},
      primaryClass={math.DG},
      url={https://arxiv.org/abs/2601.14058}, 
      note={Preprint: \url{https://arxiv.org/abs/2601.14058}}
}

@misc{GRZ26+,
      title={{PDE} aspects of the dynamical optimal transport in the {L}orentzian setting}, 
      author={Gigli, Nicola and Rott, Felix and Zanardini, Matteo},
      year={2026},
      eprint={2601.13167},
      archivePrefix={arXiv},
      primaryClass={math.AP},
      url={https://arxiv.org/abs/2601.13167}, 
      note={Preprint: \url{https://arxiv.org/abs/2601.13167}}
}

@article {CM16,
    AUTHOR = {Chudtong, Mantana and Maneesawarng, Chaiwat},
     TITLE = {An upper length estimate for curves in {${\rm CAT}(K)$}
              spaces},
   JOURNAL = {East-West J. Math.},
  FJOURNAL = {East-West Journal of Mathematics},
    VOLUME = {18},
      YEAR = {2016},
    NUMBER = {1},
     PAGES = {1--26},
      ISSN = {1513-489X},
   MRCLASS = {53C23 (51M16 53A35)},
  MRNUMBER = {3642567},
MRREVIEWER = {Barry\ Minemyer},
}

@article {AB98,
    AUTHOR = {Alexander, Stephanie B. and Bishop, Richard L.},
     TITLE = {The {F}ary-{M}ilnor theorem in {H}adamard manifolds},
   JOURNAL = {Proc. Amer. Math. Soc.},
  FJOURNAL = {Proceedings of the American Mathematical Society},
    VOLUME = {126},
      YEAR = {1998},
    NUMBER = {11},
     PAGES = {3427--3436},
      ISSN = {0002-9939,1088-6826},
   MRCLASS = {53C22 (53C40 57M25)},
  MRNUMBER = {1459103},
MRREVIEWER = {Wolfgang\ K\"uhnel},
       DOI = {10.1090/S0002-9939-98-04423-2},
       URL = {https://doi.org/10.1090/S0002-9939-98-04423-2},
}

@article {Res68,
    AUTHOR = {Reshetnyak, Yu.\ G.},
     TITLE = {Non-expansive maps in a space of curvature no greater than
              {$K$}},
   JOURNAL = {Sibirsk. Mat. \v Z.},
  FJOURNAL = {Akademija Nauk SSSR. Sibirskoe Otdelenie. Sibirski\u i\
              Matemati\v ceski\u i\ \v Zurnal},
    VOLUME = {9},
      YEAR = {1968},
     PAGES = {918--927},
      ISSN = {0037-4474},
   MRCLASS = {53.95},
  MRNUMBER = {244922},
MRREVIEWER = {L.\ W.\ Green},
}

@article {BCO06,
    AUTHOR = {Barros, Manuel and Caballero, Magdalena and Ortega, Miguel},
     TITLE = {Massless particles in warped three spaces},
   JOURNAL = {Internat. J. Modern Phys. A},
  FJOURNAL = {International Journal of Modern Physics A. Particles and
              Fields. Gravitation. Cosmology. Astrophysics. Accelerator
              Physics},
    VOLUME = {21},
      YEAR = {2006},
    NUMBER = {3},
     PAGES = {461--473},
      ISSN = {0217-751X,1793-656X},
   MRCLASS = {53C50 (53C80)},
  MRNUMBER = {2197347},
MRREVIEWER = {Angel\ Ferr\'andez},
       DOI = {10.1142/S0217751X06025559},
       URL = {https://doi.org/10.1142/S0217751X06025559},
}

@article {CFO07,
    AUTHOR = {Cabrerizo, Jose L. and Fernandez, Manuel and Ortega, Miguel},
     TITLE = {Massless particles in three-dimensional {L}orentzian warped
              products},
   JOURNAL = {J. Math. Phys.},
  FJOURNAL = {Journal of Mathematical Physics},
    VOLUME = {48},
      YEAR = {2007},
    NUMBER = {1},
     PAGES = {012901, 12},
      ISSN = {0022-2488,1089-7658},
   MRCLASS = {53C80 (53C50 58E30)},
  MRNUMBER = {2292611},
       DOI = {10.1063/1.2409522},
       URL = {https://doi.org/10.1063/1.2409522},
}

@article {Mil50,
    AUTHOR = {Milnor, J. W.},
     TITLE = {On the total curvature of knots},
   JOURNAL = {Ann. of Math. (2)},
  FJOURNAL = {Annals of Mathematics. Second Series},
    VOLUME = {52},
      YEAR = {1950},
     PAGES = {248--257},
      ISSN = {0003-486X},
   MRCLASS = {56.0X},
  MRNUMBER = {37509},
MRREVIEWER = {J.\ Nielsen},
       DOI = {10.2307/1969467},
       URL = {https://doi.org/10.2307/1969467},
}

@article {Gra89,
    AUTHOR = {Grayson, Matthew A.},
     TITLE = {The shape of a figure-eight under the curve shortening flow},
   JOURNAL = {Invent. Math.},
  FJOURNAL = {Inventiones Mathematicae},
    VOLUME = {96},
      YEAR = {1989},
    NUMBER = {1},
     PAGES = {177--180},
      ISSN = {0020-9910,1432-1297},
   MRCLASS = {53A04 (58F25)},
  MRNUMBER = {981740},
MRREVIEWER = {Joel\ L.\ Weiner},
       DOI = {10.1007/BF01393973},
       URL = {https://doi.org/10.1007/BF01393973},
}

@article {CS24,
    AUTHOR = {Coiculescu, Matei P. and Schwartz, Richard Evan},
     TITLE = {The affine shape of a figure-eight under the curve shortening
              flow},
   JOURNAL = {J. Differential Geom.},
  FJOURNAL = {Journal of Differential Geometry},
    VOLUME = {127},
      YEAR = {2024},
    NUMBER = {3},
     PAGES = {945--968},
      ISSN = {0022-040X,1945-743X},
   MRCLASS = {53E99 (53A04)},
  MRNUMBER = {4773171},
MRREVIEWER = {Anders\ Linn\'er},
       DOI = {10.4310/jdg/1721071494},
       URL = {https://doi.org/10.4310/jdg/1721071494},
}

@article {Jee84,
    AUTHOR = {Jee, Dong Jin},
     TITLE = {Gauss--{B}onnet formula for general {L}orentzian surfaces},
   JOURNAL = {Geom. Dedicata},
  FJOURNAL = {Geometriae Dedicata},
    VOLUME = {15},
      YEAR = {1984},
    NUMBER = {2},
     PAGES = {215--231},
      ISSN = {0046-5755,1572-9168},
       DOI = {10.1007/BF00147645},
       URL = {https://doi.org/10.1007/BF00147645},
}
\bibliographystyle{abbrv}

\end{document}